\documentclass[12pt,reqno]{amsart}

\usepackage{fullpage}
\usepackage{bbm}
\usepackage{amssymb}
\usepackage{mathtools}
\usepackage{mathrsfs,,todonotes}
\usepackage[all]{xy}
\usepackage{hyperref}

\hypersetup{
	colorlinks=true,
	linkcolor=blue,
	filecolor=magenta,
	urlcolor=cyan,
}

\usepackage{tikz}
\usetikzlibrary{decorations.pathreplacing}

\usepackage{ytableau}

\def\GL{\mbox{\rm GL}}
\def\SL{\mbox{\rm SL}}

\def\Ad{\operatorname{Ad}}

\def\Ann{\operatorname{Ann}}
\def\Aff{\operatorname{Aff}}
\def\chr{\operatorname{char}}
\def\col{\operatorname{col}}
\def\Dist{\operatorname{Dist}}
\def\End{\operatorname{End}}

\def\Hom{\operatorname{Hom}}
\def\Lie{\operatorname{Lie}}

\def\gr{\operatorname{gr}}

\def\max{\operatorname{max}}
\def\Rep{\operatorname{Rep}}
\def\row{\operatorname{row}}
\def\Lie{\operatorname{Lie}}
\def\Tab{\operatorname{Tab}}

\def\tr{\operatorname{tr}}

\def\b{{\mathfrak b}}
\def\g{{\mathfrak g}}
\def\n{{\mathfrak n}}

\def\t{{\mathfrak t}}

\def\sl{{\mathfrak{sl}}}

\def\pgl{{\mathfrak{pgl}}}

\def\C{{\mathbb C}}

\def\F{{\mathbb F}}

\def\N{{\mathbb N}}
\def\O{{\mathbb O}}
\def\Q{{\mathbb Q}}

\def\Z{{\mathbb Z}}

\def\sC{{\mathscr C}}

\def\cA{{\mathcal A}}
\def\cB{{\mathcal B}}
\def\cI{{\mathcal I}}
\def\cS{{\mathcal S}}
\def\cV{{\mathcal V}}
\def\cVA{{\cV\cA}}

\newcommand{\twopartdef}[4]{\left\{
	\begin{array}{ll}
		#1 & \mbox{if } #2 \\
		#3 & \mbox{if } #4
	\end{array}
	\right.}

\newcommand{\arxiv}[1]{{\tt arXiv:#1}}

\newcommand{\bK}{{\mathbb K}}
\newcommand{\bQ}{{\mathbb Q}}
\renewcommand{\g}{{\mathfrak g}}
\newcommand{\gl}{{\mathfrak g}{\mathfrak l}}
\newcommand{\m}{{\mathfrak m}}
\newcommand{\hgt}{{{\mbox{\rm ht}}}}

\numberwithin{equation}{section}

\newtheorem{theorem}{Theorem}[section]
\newtheorem{prop}[theorem]{Proposition}
\newtheorem{lemma}[theorem]{Lemma}
\newtheorem{cor}[theorem]{Corollary}
\newtheorem*{conj}{Conjecture}

\theoremstyle{definition}

\newtheorem{exa}{Example}

\theoremstyle{remark}
\newtheorem*{rmk}{Remark}

\begin{document}
	\title
	{Two problems in the representation theory of reduced enveloping algebras}
\author{Matthew Westaway}
\address{Department of Mathematical Sciences,
University of Bath, Claverton Down,
Bath, BA2 7AY,
UK}
\email{mpwestaway@gmail.com}
\date{\today}

\begin{abstract}
In this paper we consider two problems relating to the representation theory of Lie algebras $\g$  of reductive algebraic groups $G$ over algebraically closed fields $\bK$ of positive characteristic $p>0$. First, we consider the tensor product of two baby Verma modules $Z_{\chi}(\lambda)\otimes Z_{\chi'}(\mu)$ and show that it has a filtration of baby Verma modules of a particular form. Secondly, we consider the minimal-dimension representations of a reduced enveloping algebra $U_\chi(\g)$ for a nilpotent $\chi\in\g^{*}$. We show that under certain assumptions in type $A$ we can obtain the minimal-dimensional modules as quotients of certain modules obtained by base change from simple highest weight modules over $\C$. 
\end{abstract}
\maketitle
\section{Introduction}\label{s: Intro}

This paper tackles two questions relating to the representation theory of Lie algebras of reductive algebraic groups in positive characteristic. In this introduction we begin by setting up some general notation, before discussing the two questions in more detail.

Let $G$ be a reductive algebraic group over an algebraically closed field $\bK$ of characteristic $p>0$ satisfying the standard hypotheses (see \cite[\S 6.3--6.4]{Janmodreps}), and let $\g$ be its Lie algebra. We pick a maximal torus $T$ of $G$ and a Borel subgroup $B$ of $G$ containing $T$, and denote their Lie algebras by $\t$ and $\b$ respectively. We then denote by $X(T)$ the character group of $T$, by $\Phi\subseteq X(T)$ the corresponding root system, and by $\Phi^{+}$ and $\Pi$ the sets of positive and simple roots in $\Phi$ corresponding to $B$. Given $\alpha\in\Phi$, we write $\g_\alpha$ for the associated root space in $\g$ and we write $\n^{+}=\bigoplus_{\alpha\in\Phi^{+}}\g_\alpha$ and $\n^{-}=\bigoplus_{\alpha\in\Phi^{+}}\g_{-\alpha}$.

Given $\chi\in\g^{*}$ we may define the reduced enveloping algebra $U_\chi(\g)$; since every simple $\g$-module is a simple $U_\chi(\g)$-module for some $\chi\in\g^{*}$, we may reduce many questions about $\g$-modules to questions about $U_\chi(\g)$-modules. Assume now that $\chi(\n^{+})=0$ (which we are usually permitted to do without loss of generality) and denote $$\Lambda_\chi=\{\lambda\in\t^{*}\mid \lambda(h)^p-\lambda(h^{[p]})=\chi(h)^p\mbox{ for all }h\in\t\}$$ (here, $x\mapsto x^{[p]}$ denotes the restricted structure on $\g$ and $\t$). Given $\lambda\in\Lambda_\chi$ we may then define the baby Verma module $Z_\chi(\lambda)=U_\chi(\g)\otimes_{U_\chi(\b)} \bK_\lambda$, where $\bK_\lambda$ is the one-dimensional $U_\chi(\b)$-module on which $\n^{+}$ acts as zero and $\t$ acts via $\lambda$. Each simple $U_\chi(\g)$-module is then a quotient of a baby Verma module.

In the important case when $\chi$ is in standard Levi form (which we don't define in this introduction, but can be found in Subsection~\ref{ss: Prelims}), each baby Verma module $Z_\chi(\lambda)$ in fact has a {\em unique} simple quotient, which we denote $L_\chi(\lambda)$. Since the baby Verma modules are finite-dimensional they all have composition series; one broad question in this area is then to determine the composition multiplicities $$[Z_\chi(\lambda):L_\chi(\mu)]$$
for $\lambda,\mu\in\Lambda_\chi$. More generally, we wish to understand the structure of the simple $U_\chi(\g)$-modules.

Let us turn now to the first problem we tackle. Given $\chi,\chi'\in\g^{*}$, we may take the tensor product of a $U_\chi(\g)$-module $M$ and a $U_{\chi'}(\g)$-module $N$ to get a $U_{\chi+\chi'}(\g)$-module $M\otimes N$. In particular, given $\lambda\in\Lambda_\chi$ and $\mu\in\Lambda_{\chi'}$ we may form the $U_{\chi+\chi'}(\g)$-module $$Z_\chi(\lambda)\otimes Z_{\chi'}(\mu).$$ The first goal of this paper is then to understand a little about the structure of this module. The main result (Theorem~\ref{thm: Filt of tens}) we prove is the following; in this statement, we enumerate $\Phi^{+}=\{\gamma_1,\ldots,\gamma_D\}$.

\begin{theorem}\label{thm: Intro: Filt of tens}
Let $\chi,\chi'\in\g^{*}$ with $\chi(\n^{+})=\chi'(\n^{+})=0$, and let $\lambda\in\Lambda_\chi$ and $\mu\in \Lambda_{\chi'}$. Then $Z_\chi(\lambda)\otimes Z_{\chi'}(\mu)$ has a $U_{\chi+\chi'
}(\g)$-module filtration in which the successive quotients are precisely the modules $$Z_{\chi+\chi'}(\lambda+\mu -b_D\gamma_D-\cdots - b_1\gamma_1) \quad \mbox{for} \quad 0\leq b_1,\ldots, b_D<p.$$
Each such module appears precisely once in the filtration for each tuple $(b_1,\ldots,b_D)\in[0,p)^{D}$.
\end{theorem}

Before moving on to the second topic of this paper, let us discuss a little bit of the motivation for this question. Firstly, we note that in the case where $\chi$ is regular nilpotent and $\chi'=-\chi$ the module $Z_\chi(\lambda)\otimes Z_{-\chi}(\mu)$ was studied briefly by Bezrukavnikov and Riche in \cite[Proposition 4.4]{BR}, where it played in a key role in their proof that the affine Hecke category acts on the principal block of $\Rep(G)$ when $p$ is greater than the Coxeter number. We also note that analogous questions regarding composition multiplicities in tensor products of modules have been studied in a great many settings, and it is sensible to explore these questions for reduced enveloping algebras as well. For example, the tensor product of two finite-dimensional simple $\GL_N(\C)$-modules (equivalently, finite-dimensional simple $\gl_N(\C)$-modules) decomposes into simple constituents according to the Littlewood-Richardson rule. For a general symmetrisable Kac-Moody algebra over $\C$ (for example, a complex semisimple Lie algebra), this was generalised by Littelmann through the use of Littelmann paths \cite{Li1,Li2}. The tensor product of certain polynomial $\GL_N(\bK)$-modules has also been studied by Brundan and Kleshchev in \cite{BKtens,BKsum}.

Let us turn now to the second subject of this paper. Note that the reductive algebraic group $G$ acts on $\g^{*}$ via the coadjoint action; in particular, we may consider coadjoint $G$-orbits in $\g^{*}$. One of the first major conjectures in the representation theory of Lie algebras in positive characteristic was due to Kac and Weisfeiler in \cite{KW}, and it posited that $$p^{\dim (G\cdot\chi)/2}\mid \dim M$$ for all $U_\chi(\g)$-modules $M$. This was proved by Premet in \cite{PrKW}. A related conjecture was made by Humphreys in \cite{Hu2} (see also Kac's comment in \cite{Kac}), which states that there always exists a $U_\chi(\g)$-module with dimension exactly $p^{\dim (G\cdot\chi)/2}$. This is proved under the standard hypotheses by Premet and Topley in \cite{PT21} (although can be deduced more easily in type $A$ using parabolic induction). Combining these two results yields the fact that the minimal dimension of a $U_\chi(\g)$-module is $p^{\dim (G\cdot\chi)/2}$; we therefore refer to $U_\chi(\g)$-modules of this dimension as {\em minimal-dimensional modules}.

Many questions remain open about the minimal-dimensional $U_\chi(\g)$-modules. For $\g=\gl_N(\bK)$ a classification of such modules was obtained in \cite{GT} (which we make significant use of in this paper), but a classification in other types remains open. The conjecture regarding minimal-dimensional modules that we focus on was made in \cite[Remark 5.9]{PT21}, but requires a little bit of set-up. We go through this in more detail in Section~\ref{s: Min Dim'l Mods}, so we only give the basic idea here.

Suppose that $\g_\C$ is a reductive Lie algebra over $\C$ with suitable (Chevalley) basis $\cB$ and let $R$ be a localisation of $\Z$ at a finite number of primes such that $p$ is not invertible in $R$. We then set $\g_R$ to be the $R$-span of $\cB$ and define the $\bK$-Lie algebra $\g_\bK=\g_R\otimes_R \bK$ (that this is a Lie algebra obviously requires that our $\cB$ be chosen appropriately). We now assume that $\g_\C$ and $\cB$ are chosen so that $\g_\bK=\g$. Given $\lambda:\t_R\to R$ (where $\t_R=\t_\C\cap\g_R$ for a Cartan subalgebra $\t_\C$ of $\g_\C$) we may form the Verma module $M_\C(\lambda)$ and its simple quotient $L_\C(\lambda)$. These are modules for the universal enveloping algebra $U(\g_\C)$, and thus for its $R$-form $U(\g_R)$. Letting $\overline{v}_\lambda$ be the natural 
generator of $L_\C(\lambda)$, we define $L_R(\lambda)=U(\g_R)\overline{v}_\lambda$ and $L_p(\lambda)=L_R(\lambda)\otimes_R \bK$. This is a $U(\g)$-module. Given $\chi\in \g^{*}$, we can define the ideal $J_\chi$ of $U(\g)$ generated by the elements $x^p-x^{[p]}-\chi(x)^p$ for $x\in\g$. Then $L_p^\chi(\lambda):=L_p(\lambda)/J_\chi L_p(\lambda)$ is a (possibly trivial) $U_\chi(\g)$-module.

A weak version of the conjecture of \cite{PT21} is as follows. To state it, recall that under the standard hypotheses there is a one-to-one correspondence between nilpotent $G_\C$-orbits in $\g_\C^{*}$ and nilpotent $G$-orbits in $\g^{*}$ (where $G_\C$ is defined to be a complex reductive algebraic group such that $\Lie(G_\C)=\g_\C$). This correspondence has the property that for a nilpotent orbit $\O_\C$ in $\g_\C^{*}$ we may choose a representative $\chi$ which maps $\g_R$ into $R$ and which has the property that $\overline{\chi}:=\chi\otimes 1\in \g^{*}$ lies in the corresponding nilpotent orbit $\O$ in $\g^{*}$. We recall furthermore that a primitive ideal $I$ in $U(\g_\C)$ is called {\em completely prime} if $U(\g_\C)/I$ is a domain, and that the associated variety $\cVA'(I)$ of $I$ is a subvariety of $[\g_\C,\g_\C]^{*}$ (defined in Subsection~\ref{ss: Assoc Var}) which is the closure of a nilpotent orbit by Joseph's irreducibility theorem \cite{Jo81,Jo85}.

\begin{conj}

Let $R=\Z[1/q\mid q \mbox{ is a bad prime for }\Phi]$. For each nilpotent coadjoint $G_\C$-orbit $\O_\C$ in $\g_\C^{*}$, there exists a representative $\chi\in \O_\C$ mapping $\g_R$ into $R$ and there exists $\widehat{\chi}\in G\cdot\overline{\chi}$ such that the following result is true:

Let $L$ be a minimal-dimensional $U_{\widehat{\chi}}(\g)$-module. Then there exists $\lambda\in\t_R^{*}$ such that $\Ann_{U(\g_\C)}(L_\C(\lambda))$ is completely prime, $\cVA'(\Ann_{U(\g_\C)}(L_\C(\lambda)))=\overline{G_\C\cdot\chi}$, and $L$ is a composition factor of $L_p^{\widehat{\chi}}(\lambda)$.

\end{conj}

To state the full version of the conjecture from \cite{PT21} we would need to introduce some notions from the theory of finite $W$-algebras, but we omit such discussion in this paper.\footnote{To be precise, one should replace the property that $\Ann_{U(\g_\C)}(L_\C(\lambda))$ is completely prime with the property that 
$\Ann_{U(\g_\C)}(L_\C(\lambda))$ is a Losev-Premet ideal, in the language of \cite{GTW}.
} In type $A$, where we focus our attention, the conjecture stated in \cite{PT21} coincides with the one stated above, so we do not lose much by considering this conjecture instead.

In the current paper we begin by making some general observations relating to this conjecture in the case when $\chi$ is in standard Levi form. In particular, we easily get the following result (Corollary~\ref{cor: surj hom simple}).

\begin{prop}\label{cor: Intro surj hom simple}
Suppose that $\chi$ is in standard Levi form and that $\lambda\in\t_R^{*}$. If $L_p^\chi(\lambda)\neq 0$ then there is a surjective homomorphism of $U_\chi(\g)$-modules $L_p^\chi(\lambda)\twoheadrightarrow L_\chi(\overline{\lambda})$.
\end{prop}

Our main result focuses exclusively on the case when $\g=\gl_N(\bK)$. In this case there is a classification of the minimal-dimensional $U_\chi(\g)$-modules due to \cite[Theorem 1.1]{GT}; furthermore, all nilpotent $\chi\in\gl_N(\bK)^{*}$ are (up to $G$-conjugacy) in standard Levi form. The main result of the second part of this paper is the following (Theorem~\ref{thm: summation}) -- we require certain assumptions in this result labelled (R1), (R2($\lambda$)) and (R3($\lambda$)), whose statements can be found in Subsections~\ref{ss: Prelim 2} and \ref{ss: Assoc Var}.

\begin{theorem}\label{thm: Intro summation}
Let $\underline{p}$ be a partition of $N$. There exist nilpotent $\chi\in\gl_N(\C)^{*}$ and $\overline{\chi}\in\gl_N(\bK)^{*}$, which coincide on $\gl_N(\Z)$ and correspond to the partition $\underline{p}$, such that the following is true:

For any minimal-dimensional $U_{\overline{\chi}}(\gl_N(\bK))$-module $L$ there exists $\lambda\in\t_\Q^{*}$ such that if $p$ is invertible in an $R$ satisfying (R1), (R2($\lambda$)) and (R3($\lambda$)) then:
\begin{enumerate}
	\item $\Ann_{U(\gl_N(\C))}(L_\C(\lambda))$ is completely prime,
	\item $\cVA'(\Ann_{U(\gl_N(\C))}(L_\C(\lambda)))=\overline{\GL_N(\C)\cdot \chi}$, and
	\item $L_p^{\overline{\chi}}(\lambda)\twoheadrightarrow L$.
\end{enumerate}

\end{theorem}

This immediately yields the analogous result for $\g=\sl_N(\bK)$.

\begin{cor}\label{cor: Intro summation}
Let $\underline{p}$ be a partition of $N$. There exist nilpotent $\chi\in\sl_N(\C)^{*}$ and $\overline{\chi}\in\sl_N(\bK)^{*}$, which coincide on $\sl_N(\Z)$ and correspond to the partition $\underline{p}$, such that the following is true:

For any minimal-dimensional $U_{\overline{\chi}}(\sl_N(\bK))$-module $L$ there exists $\lambda\in\t_\Q^{*}$ such that if $p$ is invertible in an $R$ satisfying (R1), (R2($\lambda$)) and (R3($\lambda$)) then:
\begin{enumerate}
	\item $\Ann_{U(\sl_N(\C))}(L_\C(\lambda))$ is completely prime,
	\item $\cVA'(\Ann_{U(\sl_N(\C))}(L_\C(\lambda)))=\overline{\SL_N(\C)\cdot \chi}$, and
	\item $L_p^{\overline{\chi}}(\lambda)\twoheadrightarrow L$.
\end{enumerate}
\end{cor}

We begin the paper in Section~\ref{s: Tensor Prods} by considering the question regarding tensor products. In Section~\ref{s: Min Dim'l Mods} we then explore the minimal-dimensional modules question. More detailed description of the layouts of these sections are given in the preamble to each.

\subsection*{Acknowledgements}
The author would like to thank Simon Goodwin for helpful discussions regarding this paper and Alexander Premet for pointing out an issue with a previous version of this paper. This author was supported during this research by a research fellowship from the Royal Commission for the Exhibition of 1851.

\section{Tensor products of baby Verma modules}\label{s: Tensor Prods}

In this section, we discuss the structure of the tensor product of two baby Verma modules. Subsection~\ref{ss: Prelims} begins by setting up some conventions and notation for this section. Subsection~\ref{ss: bbVma over n} then considers the structure of baby Verma modules over a particular subalgebra of $U_\chi(\g)$, which is then applied in Subsection~\ref{ss: Tens of bbVmas} to determine a filtration of the tensor product of baby Verma modules. Finally, we conclude in Subsection~\ref{ss: Graded} by considering the graded version of these results (in the sense of \cite{Janmodreps}). 

\subsection{Preliminaries}\label{ss: Prelims}

Let $G$ be a reductive algebraic group over an algebraically closed field $\bK$ of characteristic $p>0$, and let $\g$ be its Lie algebra. Assume that $G$ satisfies the standard hypotheses, i.e. that (A) the derived subgroup of $G$ is simply connected, (B) the prime $p$ is good for $G$, and (C) there exists a non-degenerate $G$-invariant symmetric bilinear form on $\g$. Let $T$ be a maximal torus of $G$ and let $B$ be a Borel subgroup of $G$ containing $T$, with corresponding Lie algebras $\t$ and $\b$. We set $X(T)$ to be the character group of $T$. Denote by $\Phi\subseteq X(T)$ the root system of $G$ corresponding to $T$, by $\Phi^{+}$ the set of positive roots corresponding to $B$, and by $\Pi$ the associated set of simple roots. Each root $\alpha\in\Phi$ is a homomorphism of algebraic groups $T\to \bK^{*}$, and it differentiates to a homomorphism of Lie algebras $\t\to\bK$ which we also denote by $\alpha$. For each root $\alpha\in\Phi$, we denote by $\g_\alpha$ the corresponding root space in $\g$. Define $\n^{+}=\bigoplus_{\alpha\in\Phi^{+}}\g_\alpha$ and $\n^{-}=\bigoplus_{\alpha\in\Phi^{+}}\g_{-\alpha}$. We thus have $\b=\t\oplus \n^{+}$ and $\g=\n^{-}\oplus \t\oplus \n^{+}$. For each $\alpha\in\Phi$ we fix a root vector $e_\alpha\in\g_\alpha$, and we fix a basis $h_1\ldots,h_d$ of $\t$; we assume that these are chosen in such a way as to satisfy the Chevalley basis relations (adapted to the reductive case). We also define $h_\alpha=[e_\alpha,e_{-\alpha}]\in\t$ for $\alpha\in\Phi$.

Enumerate $\Pi=\{\alpha_1,\ldots,\alpha_d\}$. Given $\gamma\in \Phi$, there exist $a_1,\ldots,a_d\in \Z$ (either all positive or all negative) such that $$\gamma=a_1\alpha_1 + \cdots + a_d\alpha_d.$$ We then set $$\hgt(\gamma)=a_1+\cdots + a_d;$$ clearly $\gamma\in\Phi^{+}$ if and only if $\hgt(\gamma)>0$. We label the positive roots $\Phi^{+}=\{\gamma_1,\ldots,\gamma_D\}$ such that $r<l$ implies $\hgt(\gamma_r)\leq\hgt(\gamma_l)$. Set $\rho=\frac{1}{2}\sum_{i=1}^{D}\gamma_i$; this naturally lies in $X(T)\otimes_{\Z}\bQ$ but under our assumptions it in fact lies in $X(T)$ itself.

Let $Y(T)$ be the cocharacter group of $T$ and let $\langle-,-\rangle:X(T)\times Y(T)\to\Z$ be the perfect pairing such that $(\lambda\circ \sigma)(t)=t^{\langle \lambda,\sigma\rangle}$ for all $\lambda\in X(T)$, $\sigma\in Y(T)$ and $t\in\bK^{*}$. The coroot of a root $\alpha\in\Phi$ is denoted $\alpha^\vee$ and we write $\Phi^{\vee}=\{ \alpha^\vee\mid \alpha\in\Phi\}\subseteq Y(T)$.  Given $\alpha\in\Phi$ we define $s_\alpha:X(T)\to X(T)$ by $s_\alpha(\lambda)=\lambda-\langle\lambda,\alpha^\vee\rangle \alpha$ and $t_\alpha:X(T)\to X(T)$ by $t_\alpha(\lambda)=\lambda + p\alpha$. The Weyl group $W$ is then defined to be the subgroup of $\End_\Z(X(T))$ generated by all $s_\alpha$ for $\alpha\in \Phi$ and the affine Weyl group $W_p$ is defined to be the subgroup of $\Aff_\Z(X(T))$ generated by the elements $s_\alpha$  and $t_\beta$ for $\alpha,\beta\in\Phi$. Both $W$ and $W_p$ act on $X(T)$ by construction; we also define a dot-action of $W$ and $W_p$ on $X(T)$ by $w\cdot\lambda=w(\lambda+\rho)-\rho$.

The Lie algebra $\g$ is equipped with a natural restricted structure $\g\to\g$ written as $x\mapsto x^{[p]}$. The universal enveloping algebra $U(\g)$ has a large central subalgebra generated by the elements $x^p-x^{[p]}$ for $x\in\g$, which is called the $p$-centre and is denoted $Z_p(\g)$. Given $\chi\in\g^{*}$, we set the {\bf reduced enveloping algebra} $U_\chi(\g)$ to be the central quotient $$U_\chi(\g)=\frac{U(\g)}{\langle x^p - x^{[p]} - \chi(x)^p \mid x\in\g\rangle}.$$ It is classical that every simple $\g$-module is a $U_\chi(\g)$-module for some $\chi\in\g^{*}$. The restricted structure on $\g$ restricts to restricted structures on $\b$, $\n^{-}$ and $\n^{+}$; we may therefore also define reduced enveloping algebras $U_\chi(\b)$, $U_\chi(\n^{-})$ and $U_\chi(\n^{+})$. Each of these is a subalgebra of $U_\chi(\g)$, and $U_\chi(\g)$ is free as a module over each of them by the PBW theorem. 

Denote $$\Lambda_\chi :=\{\lambda\in\t^{*}\mid \lambda(h)^p-\lambda(h^{[p]})=\chi(h)^p \mbox{ for all }h\in\t\}.$$ For any $\lambda\in \t^{*}$ we may define a one-dimensional $\b$-module $\bK_\lambda$ on which $\n^{+}$ acts as zero and $\t$ acts via $\lambda$. Assuming that $\chi(\n^{+})=0$, this $\b$-module extends to a $U_\chi(\b)$-module if and only if $\lambda\in\Lambda_\chi$. Given such $\lambda$, we may then define the {\bf baby Verma module} corresponding to $\lambda$ as $Z_\chi(\lambda):=U_\chi(\g)\otimes_{U_\chi(\b)} \bK_\lambda$. This is a $p^D$-dimensional $U_\chi(\g)$-module, and every simple $U_\chi(\g)$-module arises as a quotient of a baby Verma module. 

We often assume $\chi(\b)=0$, in which case $$\Lambda_\chi=\Lambda_0=X(T)/pX(T).$$ At times, we make the further assumption that $\chi$ has (weak) standard Levi form, i.e. that there exists $I\subseteq \Pi$ such that $\chi(\b)=0$ and $$\chi(e_{-\alpha})=\twopartdef{=0}{\alpha\in\Phi^{+}\setminus I,}{\neq 0}{\alpha\in I.}$$ (We omit the ``weak'' if $\chi(e_{-\alpha})=1$ for all $\alpha\in I$, though such property will not have a meaningful effect on the representation theory.) When $\chi$ has (weak) standard Levi form, each baby Verma module $Z_\chi(\lambda)$ has a unique simple quotient, which we denote by $L_\chi(\lambda)$. As already observed, each simple $U_\chi(\g)$-module is of the form $L_\chi(\lambda)$ for some $\lambda\in \Lambda_0$.

Define by $W_I$ the subgroup of $W$ generated by those $s_\alpha$ with $\alpha\in \Phi\cap\Z I$ and by $W_{I,p}$ the subgroup of $W_p$ generated by those $s_\alpha$ and $t_\beta$ with $\alpha,\beta\in\Phi\cap \Z I$.

\subsection{Baby Verma modules as $U_0(\n^{+})$-modules}\label{ss: bbVma over n}

Assume that $\chi(\n^{+})=0$. The baby Verma modules $Z_\chi(\lambda)$, for $\lambda\in \Lambda_\chi$, are $U_\chi(\g)$-modules and thus may be restricted to $U_0(\n^{+})$-modules. In this subsection, we explore the structure of baby Verma modules when viewed as $U_0(\n^{+})$-modules in this way.

Recall that $Z_\chi(\lambda)$ has a basis consisting of elements $$e_{-\gamma_D}^{a_D}\cdots e_{-\gamma_1}^{a_1}z_\lambda \qquad \mbox{for} \qquad 0\leq a_1,\ldots,a_D<p,$$ where $z_\lambda:=1\otimes 1\in Z_\chi(\lambda)$. We call this the ``monomial basis'' of $Z_\chi(\lambda)$, and we define $$\hgt(e_{-\gamma_D}^{a_D}\cdots e_{-\gamma_1}^{a_1}z_\lambda)=-\sum_{i=1}^{D} a_i\hgt(\gamma_i).$$ In particular, we have $\hgt(z_\lambda)=0$. For ease of notation, set $V=Z_\chi(\lambda)$. For $m\in\Z$, we define $$V_{\geq -m}=\bK\mbox{--span}\{ e_{-\gamma_D}^{a_D}\cdots e_{-\gamma_1}^{a_1}z_\lambda\mid \hgt(e_{-\gamma_D}^{a_D}\cdots  e_{-\gamma_1}^{a_1}z_\lambda)\geq -m  \}\subseteq V.$$

\begin{lemma}\label{lem: neg root Z}
Fix $m\in\N$. Let $\gamma_r\in\Phi^{+}$ and let $ e_{-\gamma_D}^{a_D}\cdots e_{-\gamma_1}^{a_1}z_\lambda\in V_{\geq -m}$. Then $$ e_{-\gamma_r}  e_{-\gamma_D}^{a_D}\cdots e_{-\gamma_1}^{a_1}z_\lambda\in V_{\geq -m -\tiny\hgt(\gamma_r)}. $$
\end{lemma}

\begin{proof}
We proceed by induction on $m$.

The base case of $m=0$ is straightforward: since $V_{\geq 0}=\bK z_\lambda$, the result follows from the definitional fact that $e_{-\gamma_r}v_\lambda\in V_{\geq -\tiny\hgt(\gamma_r)}$.

For the induction step, let us assume that the result holds for all $k\leq m$.  We prove the inductive step by reverse induction on $r$. The base case is $r=D$, in which case we have for $e_{-\gamma_D}^{a_D}\cdots  e_{-\gamma_1}^{a_1}z_\lambda\in V_{\geq -(m+1)}$ that $$e_{\gamma_r} e_{-\gamma_D}^{a_D}\cdots  e_{-\gamma_1}^{a_1}z_\lambda= \twopartdef{e_{-\gamma_D}^{a_D+1}\cdots  e_{-\gamma_1}^{a_1}z_\lambda}{a_D<p-1,}{\chi(e_{-\gamma_D})^p e_{-\gamma_{D-1}}^{a_{D-1}}\cdots  e_{-\gamma_1}^{a_1}z_\lambda}{a_D={p-1}.}$$ In the former of these cases, we have $\hgt(e_{-\gamma_D}^{a_D+1}\cdots  e_{-\gamma_1}^{a_1}z_\lambda)= \hgt(e_{-\gamma_D}^{a_D}\cdots e_{-\gamma_1}^{a_1}z_\lambda)-\hgt(\gamma_D)\geq -(m+1)-\hgt(\gamma_r)$. In the latter, either $\chi(e_{-\gamma_D})=0$ -- in which case the result is trivial -- or $\chi(e_{-\gamma_D})\neq 0$ and $\hgt(e_{-\gamma_{D-1}}^{a_{D-1}}\cdots e_{-\gamma_1}^{a_1}z_\lambda)= \hgt(e_{-\gamma_{D}}^{a_{D}}\cdots  e_{-\gamma_1}^{a_1}z_\lambda) + (p-1)\hgt(\gamma_D)\geq  -(m+1) + (p-1)\hgt(\gamma_D) \geq -(m+1) - \hgt(\gamma_r)$. Either way, the base case for the induction on $r$ thus holds.

For the inductive step, suppose that $e_{-\gamma_i}V_{\geq -(m+1)}\subseteq V_{-(m+1)-\tiny\hgt(\gamma_i)}$ for all $i> r$. Let $e_{-\gamma_D}^{a_D}\cdots  e_{-\gamma_1}^{a_1}z_\lambda\in V_{\geq -(m+1)}$;
i.e., let $\hgt(e_{-\gamma_D}^{a_D}\cdots e_{-\gamma_1}^{a_1}z_\lambda)\geq -(m+1)$. If $\hgt(e_{-\gamma_D}^{a_D}\cdots e_{-\gamma_1}^{a_1}z_\lambda)\geq -m$ then the result follows by the inductive assumption; we thus assume $\hgt(e_{-\gamma_D}^{a_D}\cdots e_{-\gamma_1}^{a_1}z_\lambda) = -(m+1)$.

We need to show $e_{-\gamma_{r}} e_{-\gamma_D}^{a_D}\cdots e_{-\gamma_1}^{a_1}z_\lambda\in V_{\geq -(m+1)-\tiny\hgt(\gamma_r)}$. Let $l$ be maximal such that $a_l\neq 0$, noting that $l\geq 1$.\footnote{This follows since $\hgt(e_{-\gamma_D}^{a_D}\cdots e_{-\gamma_1}^{a_1}z_\lambda)= -(m+1)<0$.} There are three cases to consider.

{\bf Case 1:} $r>l$.

We have $\hgt(e_{-\gamma_r}e_{-\gamma_l}^{a_l}\cdots  e_{-\gamma_1}^{a_1}z_\lambda)=-(m+1)-\hgt(\gamma_r)$ by definition. The result thus holds in this case.

{\bf Case 2:} $r=l$.

In this case we have $$e_{-\gamma_r}e_{-\gamma_D}^{a_D}\cdots  e_{-\gamma_1}^{a_1}z_\lambda=\twopartdef{e_{-\gamma_l}^{a_l+1}\cdots  e_{-\gamma_1}^{a_1}z_\lambda}{a_l<p-1,}{\chi(e_{-\gamma_l})^p e_{-\gamma_{l-1}}^{a_{l-1}}\cdots e_{-\gamma_1}^{a_1}z_\lambda}{a_{l}=p-1.}$$ As in the base case for the induction on $r$, we get $e_{-\gamma_l}e_{-\gamma_D}^{a_D}\cdots  e_{-\gamma_1}^{a_1}z_\lambda\in V_{\geq -(m+1)-\tiny\hgt(\gamma_l)}$.

{\bf Case 3:} $r<l$. We split this case into two further cases.

{\bf Case 3(a):} $\gamma_l+\gamma_r\notin \Phi^{+}$.

In this case we have $[e_{-\gamma_r},e_{-\gamma_l}]=0$ and thus $$ e_{-\gamma_r} e_{-\gamma_l}^{a_l}\cdots  e_{-\gamma_1}^{a_1}z_\lambda = e_{-\gamma_l} e_{-\gamma_r} e_{-\gamma_l}^{a_l-1}\cdots  e_{-\gamma_1}^{a_1}z_\lambda.$$ Note that $e_{-\gamma_l}^{a_l-1}\cdots  e_{-\gamma_1}^{a_1}z_\lambda\in V_{\geq -(m+1)+\tiny\hgt(\gamma_l)}$ and so $e_{-\gamma_r}e_{-\gamma_l}^{a_l-1}\cdots  e_{-\gamma_1}^{a_1}z_\lambda\in V_{\geq -(m+1)+\tiny\hgt(\gamma_l) - \tiny\hgt(\gamma_r)}$ by the inductive assumption for $m$. Since $r<l$ we have $-(m+1)+\hgt(\gamma_l) - \hgt(\gamma_r)\geq -(m+1)$, and thus $e_{-\gamma_l} e_{-\gamma_r} e_{-\gamma_l}^{a_l-1}\cdots  e_{-\gamma_1}^{a_1}z_\lambda\in V_{\geq -(m+1)+\tiny\hgt(\gamma_l)-\tiny\hgt(\gamma_r)-\tiny\hgt(\gamma_l)}= V_{-(m+1)-\tiny\hgt(\gamma_r)}$ by the inductive assumption on $r$, as required.

{\bf Case 3(b):} $\gamma_l+\gamma_r\in \Phi^{+}$.

In this case, there exists $C_{r,l}\in\bK$ such that $[e_{-\gamma_r},e_{-\gamma_l}]=C_{r,l} e_{-\gamma_l-\gamma_r}$. Then $$e_{-\gamma_r} e_{-\gamma_l}^{a_l}\cdots  e_{-\gamma_1}^{a_1}z_\lambda= e_{-\gamma_l}e_{-\gamma_r} e_{-\gamma_l}^{a_l-1}\cdots  e_{-\gamma_1}^{a_1}z_\lambda + C_{r,l}e_{-\gamma_l-\gamma_r} e_{-\gamma_l}^{a_l-1}\cdots  e_{-\gamma_1}^{a_1}z_\lambda.$$ The first part of this sum lies in $V_{\geq -(m+1)-\tiny\hgt(\gamma_r)}$ as in Case 3(a). For the second part, note that $\hgt(\gamma_l+\gamma_r)=\hgt(\gamma_l)+\hgt(\gamma_r)$ (so $\gamma_l+ \gamma_r = \gamma_{t}$ for some $t>l$); therefore, $\hgt(e_{-\gamma_l-\gamma_r} e_{-\gamma_l}^{a_l-1}\cdots e_{-\gamma_1}^{a_1}z_\lambda) = -\hgt(\gamma_l)-\hgt(\gamma_r) + \hgt(e_{-\gamma_l}^{a_l-1}\cdots  e_{-\gamma_1}^{a_1}z_\lambda)=-(m+1)-\hgt(\gamma_r)$, as required. 

The inductive step therefore holds for the induction on $r$, and thus also for the induction on $m$.

\end{proof}

The associative algebra $U_0(\n^{+})$ is a Hopf algebra; the augmentation ideal of $U_0(\n^{+})$ is then defined to be the kernel of the counit $\varepsilon$. We denote it by $I^+$; note that it is precisely the ideal generated by the elements $e_{\alpha}$ for $\alpha\in\Phi^{+}$.

\begin{prop}\label{prop: pos root Z}
Let $m\in\N$. Then $V_{\geq -m}$ is a $U_0(\n^{+})$-module and the $U_0(\n^{+})$-action on $V_{\geq -m}/V_{\geq -m+1}$ is trivial. In other words, $I^+ V_{\geq - m}\subseteq V_{\geq -m+1}$.
\end{prop}

\begin{proof}
Under our assumptions, $I^+$ is generated by the root vectors $e_{\alpha_i}$ for $i=1,\ldots,d$. It is therefore sufficient to show that $e_{\alpha_i}V_{\geq - m}\subseteq V_{\geq -m+1}$ for all $i=1,\ldots,d$.

We proceed by induction on $m$. For the base case of $m=0$, we note that $V_{\geq 0}=\bK z_\lambda$. This means in particular that $e_{\alpha_i} z_\lambda=0$ for all $i=1,\ldots,d$, and thus we have $e_{\alpha_i} V_{\geq 0} \subseteq V_{\geq 1}=\{0\}$.

For the inductive step, suppose that the result holds for all $k\leq m$. We must prove that $e_{\alpha_i} e_{-\gamma_D}^{a_D}\cdots e_{-\gamma_1}^{a_1}z_\lambda \in V_{\geq -m}$ whenever $\hgt(e_{-\gamma_D}^{a_D}\cdots  e_{-\gamma_1}^{a_1}z_\lambda)\geq -(m+1)$. The result follows from the inductive assumption if $\hgt(e_{-\gamma_D}^{a_D}\cdots  e_{-\gamma_1}^{a_1}z_\lambda)\geq -m$; we therefore assume that $\hgt(e_{-\gamma_D}^{a_D}\cdots  e_{-\gamma_1}^{a_1}z_\lambda) = -(m+1)$.

Let $l$ be maximal such that $a_l\neq 0$ (so $l\geq 1$). There are three cases to consider.

{\bf Case 1:} $\gamma_l - \alpha_i \notin \Phi^{+}\cup\{0\}$.

In this case we have $[e_{\alpha_i},e_{-\gamma_l}]=0$, and so $$ e_{\alpha_i} e_{-\gamma_l}^{a_l}\cdots  e_{-\gamma_1}^{a_1}z_\lambda = e_{-\gamma_l}e_{\alpha_i} e_{-\gamma_l}^{a_l-1}\cdots e_{-\gamma_1}^{a_1}z_\lambda.$$ Since $\hgt(e_{-\gamma_l}^{a_l-1}\cdots  e_{-\gamma_1}^{a_1}z_\lambda)=-(m+1)+\hgt(\gamma_l)\geq -m$, by induction we have that $$e_{\alpha_i}e_{-\gamma_l}^{a_l-1}\cdots  e_{-\gamma_1}^{a_1}z_\lambda\in V_{\geq -(m+1)+\tiny\hgt(\gamma_l)+\tiny\hgt(\alpha_i)}=V_{\geq -m+\tiny\hgt(\gamma_l)}.$$ By Lemma~\ref{lem: neg root Z} we thus have $e_{-\gamma_l}e_{\alpha_i} e_{-\gamma_l}^{a_l-1}\cdots  e_{-\gamma_1}^{a_1}z_\lambda\in V_{\geq -m+\tiny\hgt(\gamma_l)-\tiny\hgt(\gamma_l)}=V_{\geq -m},$ as required.

{\bf Case 2:} $\gamma_l = \alpha_i$. 

In this setting, we have $[e_{\alpha_i},e_{-\gamma_l}]=h_{\alpha_i}=: h_i$, and therefore 
\begin{equation*}
	\begin{split}
		e_{\alpha_i} e_{-\gamma_l}^{a_l}\cdots  e_{-\gamma_1}^{a_1}z_\lambda  = &  e_{-\gamma_l}e_{\alpha_i} e_{-\gamma_l}^{a_l-1}\cdots e_{-\gamma_1}^{a_1}z_\lambda +  h_i e_{-\gamma_l}^{a_l-1}\cdots  e_{-\gamma_1}^{a_1}z_\lambda \\
		= & e_{-\gamma_l}e_{\alpha_i} e_{-\gamma_l}^{a_l-1}\cdots e_{-\gamma_1}^{a_1}z_\lambda \\ & +  (\lambda-(a_l-1)\gamma_l - \cdots - a_1\gamma_1)(h_i) e_{-\gamma_l}^{a_l-1}\cdots  e_{-\gamma_1}^{a_1}z_\lambda.
	\end{split}
\end{equation*}
The former summand lies in $V_{\geq -m}$ as in Case 1, while the latter summand is either zero or has height $\hgt(e_{-\gamma_l}^{a_l}\cdots e_{-\gamma_2}^{a_2} e_{-\gamma_1}^{a_1}z_\lambda)+\hgt(\gamma_l)= -(m+1)+\hgt(\gamma_l)\geq -m$ as required.

{\bf Case 3:} $\gamma_l - \alpha_i \in \Phi^{+}$.

In this case, there exists $C_{l,i}\in\bK$ such that $[e_{\alpha_i},e_{-\gamma_l}]= C_{l,i} e_{-\gamma_l+\alpha_i}$. We then have $$ e_{\alpha_i} e_{-\gamma_l}^{a_l}\cdots e_{-\gamma_1}^{a_1}z_\lambda = e_{-\gamma_l}e_{\alpha_i} e_{-\gamma_l}^{a_l-1}\cdots  e_{-\gamma_1}^{a_1}z_\lambda + C_{l,i}  e_{-\gamma_l+\alpha_i}e_{-\gamma_l}^{a_l-1}\cdots  e_{-\gamma_1}^{a_1}z_\lambda.$$ As in Case 1, the former of the terms in this sum lies in $V_{\geq -m}$. For the latter, note that $\hgt(e_{-\gamma_l}^{a_l-1}\cdots  e_{-\gamma_1}^{a_1}z_\lambda)=-(m+1)+\hgt(\gamma_l)$; by Lemma~\ref{lem: neg root Z} we therefore have $$e_{-\gamma_l+\alpha_i}e_{-\gamma_l}^{a_l-1}\cdots e_{-\gamma_1}^{a_1}z_\lambda\in V_{\geq -(m+1)+\tiny\hgt(\gamma_l) -\tiny\hgt(\gamma_l-\alpha_i)}=V_{\geq -m}.$$ This proves the induction step and thus the result.
\end{proof}

\begin{cor}\label{cor: filt}
As a $U_0(\n^{+})$-module, $Z_\chi(\lambda)$ has a filtration $$0=U_0\subseteq U_1\subseteq U_2\subseteq \cdots \subseteq  U_{p^{\dim\n^{\tiny +}}}=Z_\chi(\lambda)$$ such that for each $i=1,\ldots,p^{\dim\n^{+}}$ there exists $e_{-\gamma_D}^{a_D}\cdots e_{-\gamma_1}^{a_1}z_\lambda\in U_i \subseteq Z_\chi(\lambda)$  such that $$U_i/U_{i-1}=\bK\mbox{-span}\{ e_{-\gamma_D}^{a_D}\cdots e_{-\gamma_1}^{a_1}z_\lambda + U_{i-1} \}.$$ Furthermore, $\n^{+}$ acts trivially on each $U_i/U_{i-1}$.
\end{cor}

\begin{proof}
By Proposition~\ref{prop: pos root Z}, $Z_\chi(\lambda)$ has a $U_0(\n^{+})$-module filtration 
\begin{equation}\label{eq: hgt filt}
	0=V_{\geq 1} \subseteq V_{\geq 0}\subseteq V_{\geq -1}\subseteq \cdots \subseteq V_{\geq -(p-1)\tiny \hgt(\rho)}=Z_\chi(\lambda)
\end{equation}
such that $U_0(\n^{+})$ acts trivially on each quotient $V_{\geq -m}/V_{\geq -m+1}$. By construction, each $V_{\geq -m}$ has a basis consisting of elements of the form $e_{-\gamma_D}^{a_D}\cdots e_{-\gamma_1}^{a_1}z_\lambda$, and therefore each $V_{\geq -m}/V_{\geq -m+1}$ has a basis consisting of elements of the form $e_{-\gamma_D}^{a_D}\cdots e_{-\gamma_1}^{a_1}z_\lambda + V_{\geq -m+1}$. Since $U_0(\n^{+})$ acts trivially on each $V_{\geq -m}/V_{\geq -m+1}$, the filtration (\ref{eq: hgt filt}) can be refined to a filtration with the properties described in the statement of the corollary.
\end{proof}

\subsection{Tensor products}\label{ss: Tens of bbVmas}

Since $U(\g)$ is a Hopf algebra, the tensor product $M\otimes N$ of two $U(\g)$-modules $M$ and $N$ can be given the structure of a $U(\g)$-module. Specifically, given $x\in\g$, $m\in M$ and $n\in N$, we have $$x\cdot(m\otimes n)= (x\cdot m)\otimes n + m\otimes (x\cdot n).$$ It is well-known (and straightforward to check) that when $M$ is a $U_\chi(\g)$-module and $N$ is a $U_{\chi'}(\g)$-module (for $\chi,\chi'\in\g^{*}$) the resulting tensor product $M\otimes N$ is in fact a $U_{\chi+\chi'}(g)$-module.

Fixing $\chi,\chi'\in\g^{*}$ with $\chi(\n^{+})=\chi'(\n^{+})=0$, and letting $\lambda\in\Lambda_\chi$ and $\mu\in \Lambda_{\chi'}$, we may therefore form the $U_{\chi+\chi'}(\g)$-module $Z_\chi(\lambda)\otimes Z_{\chi'}(\mu)$. Since we are now considering two baby Verma modules, we modify our notation from above a little bit in order to differentiate them. Specifically, we still set $z_\lambda$ to be the generator $1\otimes 1$ of $Z_\chi(\lambda)$, but we now denote by $u_\mu$ the generator $ 1\otimes 1$ of $Z_{\chi'}(\mu)$. The module $Z_\chi(\lambda)\otimes Z_{\chi'}(\mu)$ then naturally has a basis consisting of the elements 
\begin{equation}\label{e: ZoZ basis 1}
(e_{-\gamma_D}^{a_D}\cdots  e_{-\gamma_1}^{a_1}z_\lambda)\otimes (e_{-\gamma_D}^{b_D}\cdots  e_{-\gamma_1}^{b_1}u_\mu)\quad \mbox{for} \quad 0\leq a_1,\ldots,a_D,b_1,\ldots,b_D<p.
\end{equation}
It will be useful, however, to work with a slightly different basis of $Z_\chi(\lambda)\otimes Z_{\chi'}(\mu)$.

\begin{lemma}\label{lem: ZoZ basis}
The $U_{\chi+\chi'}(\g)$-module $Z_\chi(\lambda)\otimes Z_{\chi'}(\mu)$ has a basis consisting of the elements \begin{equation}\label{e: basis} e_{-\gamma_D}^{c_D}\cdots  e_{-\gamma_1}^{c_1}(z_\lambda\otimes e_{-\gamma_D}^{d_D}\cdots e_{-\gamma_1}^{d_1}u_\mu)\quad \mbox{for}\quad 0\leq c_1,\ldots,c_D,d_1,\ldots,d_D<p.
\end{equation}
\end{lemma}

\begin{proof}
Since $Z_\chi(\lambda)\otimes Z_{\chi'}(\mu)$ has dimension $p^{2\dim\n^{+}}$, it is enough to show that the elements described in (\ref{e: basis}) span $Z_\chi(\lambda)\otimes Z_{\chi'}(\mu)$.
To do so, it is sufficient to show that each element in the basis described in (\ref{e: ZoZ basis 1}) can be written as a linear combination of elements of the form described in (\ref{e: basis}). Using the notation from (\ref{e: ZoZ basis 1}), we prove this by induction on $-\hgt(e_{-\gamma_D}^{a_D}\cdots  e_{-\gamma_1}^{a_1}z_\lambda)$. It is clear when $-\hgt(e_{-\gamma_D}^{a_D}\cdots  e_{-\gamma_1}^{a_1}z_\lambda)=0$; suppose that it is true whenever $-\hgt(e_{-\gamma_D}^{a_D}\cdots e_{-\gamma_1}^{a_1}z_\lambda)<m$. If $-\hgt(e_{-\gamma_D}^{a_D}\cdots  e_{-\gamma_1}^{a_1}z_\lambda)=m$ and $l$ is maximal such that $a_l\neq 0$ then 
\begin{equation*}
	\begin{split}
		(e_{-\gamma_l}^{a_l}\cdots e_{-\gamma_1}^{a_1}z_\lambda)\otimes (e_{-\gamma_D}^{b_D}\cdots e_{-\gamma_1}^{b_1}u_\mu)  = &  e_{-\gamma_l}((e_{-\gamma_l}^{a_l-1}\cdots e_{-\gamma_1}^{a_1}z_\lambda)\otimes (e_{-\gamma_D}^{b_D}\cdots e_{-\gamma_1}^{b_1}u_\mu)) \\ &  - (e_{-\gamma_l}^{a_l-1}\cdots e_{-\gamma_1}^{a_1}z_\lambda)\otimes (e_{-\gamma_l} e_{-\gamma_D}^{b_D}\cdots e_{-\gamma_1}^{b_1}u_\mu).
	\end{split}
\end{equation*}
By induction, and using the fact that $U_{\chi+\chi'}(\n^{-})$ and $U_{\chi'}(\n^{-})$ are algebras with bases consisting of the elements $e_{-\gamma_D}^{c_D}\cdots e_{-\gamma_1}^{c_1}$ for $0\leq c_i < p$, each part of this sum can be rewritten in terms of elements of the form (\ref{e: basis}). This completes the induction step and thus proves the result.

\end{proof}

Since $\bK z_\lambda$ is a $U_\chi(\b)$-module and since we may restrict $Z_{\chi'}(\mu)$ to a $U_{\chi'}(\b)$-module, we may take the tensor product of these modules to form the $U_{\chi+\chi'}(\b)$-module $\bK z_\lambda \otimes Z_{\chi'}(\mu)$. We may therefore define a $U_{\chi+\chi'}(\g)$-module $$ V_{\chi,\chi'}(\lambda,\mu):= U_{\chi+\chi'}(\g)\otimes_{U_{\chi+\chi'}(\b)}(\bK z_\lambda\otimes Z_{\chi'}(\mu)).$$ Furthermore, Frobenius reciprocity implies that for all $U_{\chi+\chi'}(\g)$-modules $M$ there is equality $$\Hom_{U_{\chi+\chi'}(\g)}(V_{\chi,\chi'}(\lambda,\mu), M) = \Hom_{U_{\chi+\chi'}(\b)}(\bK z_\lambda\otimes Z_{\chi'}(\mu), M).$$ In particular, $$\Hom_{U_{\chi+\chi'}(\g)}(V_{\chi,\chi'}(\lambda,\mu), Z_\chi(\lambda)\otimes Z_{\chi'}(\mu)) = \Hom_{U_{\chi+\chi'}(\b)}(\bK z_\lambda\otimes Z_{\chi'}(\mu), Z_\chi(\lambda)\otimes Z_{\chi'}(\mu)).$$ The natural embedding $\bK z_\lambda\otimes Z_{\chi'}(\mu) \hookrightarrow Z_\chi(\lambda)\otimes Z_{\chi'}(\mu)$ therefore induces a $U_{\chi+\chi'}(\g)$-module homomorphism 
$$ \Psi:V_{\chi,\chi'}(\lambda,\mu) \to Z_\chi(\lambda)\otimes Z_{\chi'}(\mu) $$ which sends $$ e_{-\gamma_D}^{a_D}\cdots e_{-\gamma_1}^{a_1}\otimes_{U_{\chi+\chi'}(\b)} (z_\lambda\otimes e_{-\gamma_D}^{b_D}\cdots e_{-\gamma_1}^{b_1}u_\mu) \mapsto e_{-\gamma_D}^{a_D}\cdots e_{-\gamma_1}^{a_1}(z_\lambda\otimes e_{-\gamma_D}^{b_D}\cdots e_{-\gamma_1}^{b_1}u_\mu).$$ It is straightforward that $V_{\chi,\chi'}(\lambda,\mu)$ has a basis consisting of the elements $$e_{-\gamma_D}^{a_D}\cdots e_{-\gamma_1}^{a_1}\otimes_{U_{\chi+\chi'}(\b)} (z_\lambda\otimes e_{-\gamma_D}^{b_D}\cdots e_{-\gamma_1}^{b_1}u_\mu) \quad \mbox{for} \quad 0\leq a_1,\ldots,a_D, b_1,\ldots,b_D<p.$$ The $U_{\chi+\chi'}(\g)$-module homomorphism $\Psi$ therefore sends a basis of $V_{\chi,\chi'}(\lambda,\mu)$ to a basis of $Z_\chi(\lambda)\otimes Z_{\chi'}(\mu)$ by Lemma~\ref{lem: ZoZ basis}; it is thus an isomorphism. We have proved the following result.

\begin{prop}\label{prop: Z=V}
The $U_{\chi+\chi'}(\g)$-modules $Z_\chi(\lambda)\otimes Z_{\chi'}(\mu)$ and $V_{\chi,\chi'}(\lambda,\mu)$ are isomorphic.
\end{prop}

With this new description of $Z_\chi(\lambda)\otimes Z_{\chi'}(\mu)$ we are able to deduce the desired filtration. Recall that $Z_{\chi'}(\mu)$ has a $U_0(\n^{+})$-module filtration 
$$0= U_0\subseteq U_1 \subseteq \cdots \subseteq U_{p^{\dim \n^{+}}}=Z_{\chi'}(\mu)$$ as in Corollary~\ref{cor: filt}. In fact, each $U_i$ is a $U_{\chi'}(\b)$-module since it has a basis of $\t$-weight vectors. We may therefore form a $U_{\chi'}(\b)$-module filtration of $\bK z_\lambda\otimes Z_{\chi'}(\mu)$ as $$0= \bK z_\lambda\otimes U_0\subseteq \bK z_\lambda\otimes U_1 \subseteq \cdots \subseteq \bK z_\lambda\otimes U_{p^{\dim \n^{+}}}=\bK z_\lambda \otimes Z_{\chi'}(\mu).$$ Recall that the induction functor from the category of $U_{\chi+\chi'}(\b)$-modules to the category of $U_{\chi+\chi'}(\g)$-modules (given by $M\mapsto U_{\chi+\chi'}(\g)\otimes_{U_{\chi+\chi'}(\b)} M$) is exact. Applying this to the above filtration therefore induces a filtration $$0=W_0 \subseteq W_1\subseteq \cdots \subseteq W_{p^{\dim\n^{+}}}= V_{\chi,\chi'}(\lambda,\mu) = Z_\chi(\lambda)\otimes Z_{\chi'}(\mu)$$ of $U_{\chi+\chi'}(\g)$-modules. Furthermore, the exactness of the induction functor implies that $$W_i/W_{i-1} \cong U_{\chi+\chi'}(\g)\otimes_{U_{\chi+\chi'}(\b)}((\bK z_\lambda\otimes U_i)/(\bK z_\lambda\otimes U_{i-1})).$$

As in Corollary~\ref{cor: filt}, there exists a monomial basis element $e_{-\gamma_D}^{b_D}\cdots e_{-\gamma_1}^{b_1} u_\mu$ of $Z_{\chi'}(\mu)$ (with $0\leq b_1,\ldots, b_D <p$) such that $$U_i/U_{i-1} = \bK\mbox{-span}\{e_{-\gamma_D}^{b_D}\cdots e_{-\gamma_1}^{b_1} u_\mu + U_{i-1}\}.$$ Therefore, $(\bK z_\lambda\otimes U_i)/(\bK z_\lambda\otimes U_{i-1})$ is a one-dimensional $U_{\chi+\chi'}(\b)$-module on which $\n^{+}$ acts as zero (by Corollary~\ref{cor: filt}) and $\t$ acts via $\lambda+\mu -b_D\gamma_D-\cdots - b_1\gamma_1.$ Therefore,  $$W_i/W_{i-1} \cong U_{\chi+\chi'}(\g)\otimes_{U_{\chi+\chi'}(\b)}((\bK z_\lambda\otimes U_i)/(\bK z_\lambda\otimes U_{i-1}))\cong Z_{\chi+\chi'}(\lambda+\mu -b_D\gamma_D-\cdots - b_1\gamma_1).$$
In other words, we have proved the following theorem.

\begin{theorem}\label{thm: Filt of tens}
Let $\chi,\chi'\in\g^{*}$ with $\chi(\n^{+})=\chi'(\n^{+})=0$, and let $\lambda\in\Lambda_\chi$ and $\mu\in \Lambda_{\chi'}$. Then $Z_\chi(\lambda)\otimes Z_{\chi'}(\mu)$ has a $U_{\chi+\chi'
}(\g)$-module filtration in which the successive quotients are precisely the modules $$Z_{\chi+\chi'}(\lambda+\mu -b_D\gamma_D-\cdots - b_1\gamma_1) \quad \mbox{for} \quad 0\leq b_1,\ldots, b_D<p.$$
Each such module appears precisely once in the filtration for each tuple $(b_1,\ldots,b_D)\in[0,p)^{D}$.
\end{theorem}

\begin{rmk}
Given $\lambda\in\Lambda_\chi$, $\mu\in\Lambda_{\chi'}$ and $(b_1,\ldots,b_D)\in [0,p)^D$, we do indeed have $\lambda+\mu-b_D\gamma_D -\cdots - b_1\gamma_1\in\Lambda_{\chi+\chi'}$, since
\begin{equation*}
	\begin{split}
		(\lambda+\mu -b_D\gamma_D-\cdots - b_1\gamma_1)&(h)^p  -  (\lambda+\mu -b_D\gamma_D-\cdots - b_1\gamma_1)(h^{[p]})
		= \\ (\lambda(h)^p-\lambda(h^{[p]})) + (\mu(h)^p-\mu(h^{[p]})) & -b_D(\gamma_D(h)^p-\gamma_D(h^{[p]}))- \cdots - b_1(\gamma_1(h)^p-\gamma_1(h^{[p]}))
		\\= \chi(h)^p + &\chi'(h)^p  =(\chi+\chi')(h)^p.
	\end{split}
\end{equation*}

\end{rmk}

\begin{exa}

Consider $\g=\sl_2(\bK)$ with the usual basis $e,h,f$, and suppose $\chr(\bK)=5$. Fix $\chi\in\g^{*}$ with $\chi(e)=\chi(h)=0$ and $\chi(f)=1$. Noting that $\t^{*}=\bK$, let $\lambda=2$ and $\mu=3$. Then $Z_\chi(\lambda)$ has basis $$v_0:=1\otimes 1,\quad v_1:=fv_0, \quad v_2:=f^2v_0,\quad v_3:=f^3v_0, \quad v_4:=f^4v_0$$ and $Z_{-\chi}(\mu)$ has basis $$w_0:=1\otimes 1,\quad w_1:=fw_0, \quad w_2:=f^2w_0,\quad w_3:=f^3w_0, \quad w_4:=f^4w_0.$$ It therefore follows that $Z_\chi(\lambda)\otimes Z_{-\chi}(\mu)$ has basis consisting of the elements $$v_i\otimes w_j \quad \mbox{for} \quad 0\leq i,j <5. $$ Applying Lemma~\ref{lem: ZoZ basis} (or checking directly) this module also has a basis consisting of the elements  $$f^i(v_0\otimes w_j)\quad \mbox{for} \quad 0\leq i,j<5.$$
We then set $$W_0=0,\quad W_1=\bK\mbox{-span}\{f^i(v_0\otimes w_0)\mid 0\leq i <5\},$$  $$W_2=\bK\mbox{-span}\{f^i(v_0\otimes w_j)\mid 0\leq i <5,\, 0\leq j\leq 1\},$$
$$W_3=\bK\mbox{-span}\{f^i(v_0\otimes w_j)\mid 0\leq i <5,\, 0\leq j\leq 2\},$$ $$W_4=\bK\mbox{-span}\{f^i(v_0\otimes w_j)\mid 0\leq i <5,\, 0\leq j\leq 3\},$$ $$W_5=\bK\mbox{-span}\{f^i(v_0\otimes w_j)\mid 0\leq i <5,\, 0\leq j\leq 4\}.$$ Then $$W_1/W_0\cong Z_0(0), \qquad W_2/W_1\cong Z_0(3), \qquad W_3/W_2\cong Z_0(1),$$ $$W_4/W_3\cong Z_0(4),\qquad W_5/W_4\cong Z_0(2).$$ Since we know the composition factors for $Z_0(k)$ over $\sl_2(\bK)$ for all $k\in\F_5$, in this case we may explicitly give the composition factors of $Z_\chi(2)\otimes Z_{-\chi}(3)$ as $$L_0(0),\,\, L_0(0),\,\, L_0(3), \,\, L_0(2),\,\, L_0(1),\,\, L_0(4),\,\, L_0(1),\,\, L_0(2).$$

\end{exa}

\subsection{Graded setting}\label{ss: Graded}

Suppose that $\chi\in\g^{*}$ is in (weak) standard Levi form, as described in Subsection~\ref{ss: Prelims}. The reduced enveloping algebra $U_\chi(\g)$ can be equipped with an $X(T)/\Z I$-grading, which is induced by setting $e_\alpha\in U_\chi(\g)_{\alpha + \Z I}$ and $\t\subseteq U_\chi(\g)_{0+\Z I}$. Define $\sC_\chi$ to be the category of $X(T)/\Z I$-graded $U_\chi(\g)$-modules $M$ with the property (X) that each homogeneous component $M_{\lambda+\Z I}$ of $M$ decomposes as a sum of $\t$ weight spaces as $$M_{\lambda+ \Z I}=\bigoplus_{\mu\in\lambda+\Z I +pX(T)} M_{\lambda+\Z I}^{d\mu}. $$ Morphisms in $\sC_\chi$ are homomorphisms of $U_\chi(\g)$-modules which preserve the $X(T)/\Z I$-grading.

Given $\lambda\in X(T)$ we define the baby Verma module $Z_\chi(\lambda)=U_\chi(\g)\otimes_{U_0(\b)}\bK_\lambda\in\sC_\chi$, where $\bK_\lambda$ is the one-dimensional $\b$-module on which $\n^{+}$ acts as zero and $\t$ acts as $d\lambda$, and where the $X(T)/\Z I$-grading on $Z_\chi(\lambda)$ is induced by setting $1\otimes 1\in Z_\chi(\lambda)_{\lambda+\Z I}$. If we forget the $X(T)/\Z I$-grading of $Z_\chi(\lambda)$ then we obtain the baby Verma module $Z_\chi(d\lambda)$ considered in the previous subsections. Each $Z_\chi(\lambda)$ has a unique simple quotient $L_\chi(\lambda)\in \sC_\chi$, and each simple object in $\sC_\chi$ is isomorphic to some $L_\chi(\lambda)$. By \cite[Proposition 2.6]{Janmodreps},
$$L_\chi(\lambda)\cong L_\chi(\mu) \quad \iff \quad Z_\chi(\lambda)\cong Z_\chi(\mu)\quad \iff \quad \lambda\in W_{I,p}\cdot \mu.$$ Denote by $\Lambda_I$ a set of representatives in $X(T)$ of the $W_{I,p}$-orbits in $X(T)$ under the dot-action.

A special case of $\sC_\chi$ occurs when $\chi=0$, in which case the objects are certain $X(T)$-graded $U_0(\g)$-modules; in fact, this category coincides with the category of finite-dimensional $G_1T$-modules, where $G_1$ is the first Frobenius kernel of $G$. Given $M\in\sC_\chi$ and $N\in\sC_{-\chi}$, we would like to define $M\otimes N\in\sC_0$. This will not be possible, however, so we must settle for a slightly weaker construction.

Specifically, we define $\widehat{\sC}_0$ to be the category of $X(T)/\Z I$-graded $U_0(\g)$-modules which satisfy property (X). There is then a natural functor $$ \sC_0\to \widehat{\sC}_0,\qquad M\mapsto \widehat{M}, $$ where $\widehat{M}$ coincides with $M$ as a $U_0(\g)$-module but has $X(T)/\Z I$-grading given by 
$$ M_{\lambda+\Z I}=\bigoplus_{\eta\in\lambda+\Z I}M_\eta.$$ For ease of notation, however, we write $\widehat{Z}_0(\lambda)$ and $\widehat{L}_0(\lambda)$ in $\widehat{\sC}_0$ in place of $\widehat{Z_0(\lambda)}$ and $\widehat{L_0(\lambda)}$. Note that the module $\widehat{L}_0(\lambda)$ is simple in $\widehat{\sC}_0$, since by \cite[Lemma 1]{Janmodreps} (see also \cite[Proposition 3.5]{GG}) $L_0(\lambda)$ remains simple when we forget the $X(T)$-grading entirely. 

\begin{prop}\label{prop: tens grad}
Let $M\in\sC_\chi$ and $N\in\sC_{-\chi}$. Then $M\otimes N$ can be given an $X(T)/\Z I$-grading such that $M\otimes N\in \widehat{\sC}_0$.
\end{prop}

\begin{proof}
Define $$(M\otimes N)_{\lambda+\Z I}:=\bigoplus_{\mu+\Z I\in X(T)/\Z I} M_{\mu+\Z I}\otimes N_{\lambda-\mu+\Z I}.$$ It is straightforward to check that this makes $M\otimes N$ into an $X(T)/\Z I$-graded $U_0(\g)$-module satisfying property (X), as required.    
\end{proof}

We may now extend Theorem~\ref{thm: Filt of tens} to the $X(T)/\Z I$-graded setting.

\begin{prop}\label{prop: grad tens filt}
Let $\chi\in\g^{*}$ be in (weak) standard Levi form, corresponding to $I\subseteq \Pi$, and let $\lambda,\mu\in X(T)$. Then $Z_\chi(\lambda)\otimes Z_{-\chi}(\mu)\in\widehat{\sC}_0$ has a filtration in which the successive quotients are precisely the modules $$\widehat{Z}_{0}(\lambda+\mu -a_D\gamma_D-\cdots - a_1\gamma_1) \quad \mbox{for} \quad 0\leq a_1,\ldots, a_D<p.$$
Each such module appears precisely once in the filtration for each tuple $(b_1,\ldots,b_D)\in[0,p)^{D}$.
\end{prop}

\begin{proof}
The follows as in Theorem~\ref{thm: Filt of tens}, paying attention to the $X(T)/\Z I$-grading at each point.
\end{proof}

This result allows us to say something about the composition series of $Z_\chi(\lambda)\otimes Z_{-\chi}(\mu)$ in $\widehat{\sC}_0$. We have already observed that the objects $\widehat{L}_0(\kappa)$ are simple in $\widehat{\sC}_0$ and all simple objects are of this form (by an argument similar to that in \cite[\S 2.5]{Janmodreps}). Furthermore, as above $$L_0(\kappa)\cong L_0(\tau) \quad \iff \quad \kappa=\tau$$ and it is clear that $$ \widehat{L}_0(\kappa)\cong \widehat{L}_0(\tau) \quad \iff \quad \kappa-\tau\in p\Z I.$$ So, in particular, non-isomorphic modules in $\sC_0$ can become isomorphic in $\widehat{\sC}_0$.

Combining this observation with Proposition~\ref{prop: grad tens filt} yields
\begin{equation}\label{e: ZoZ Comp1}
\begin{split}
	[Z_\chi(\lambda)\otimes Z_{-\chi}(\mu):\widehat{L}_0(\kappa)] & =\sum_{0\leq a_i <p} [\widehat{Z}_0(\lambda+\mu -a_D \gamma_D -\cdots - a_1\gamma_1): \widehat{L}_0(\kappa)] \\ & = \sum_{0\leq a_i <p}\sum_{\gamma\in  \Z I} [Z_0(\lambda+\mu -a_D \gamma_D -\cdots - a_1\gamma_1): L_0(\kappa+p\gamma)]
\end{split}
\end{equation}
for all $\lambda,\mu,\kappa\in X(T)$.

Alternatively, if $Z_\chi(\lambda)$ has composition series with composition factors $L_\chi(\sigma_1),\ldots,L_\chi(\sigma_s)$ and $Z_{-\chi}(\mu)$ has composition series with composition factors $L_{-\chi}(\tau_1),\ldots,L_{-\chi}(\tau_t)$, then $Z_\chi(\lambda)\otimes Z_{-\chi}(\mu)$ has a filtration with successive quotients $L_{\chi}(\sigma_i)\otimes L_{-\chi}(\tau_j)$. In particular, we have equality $$[Z_\chi(\lambda)\otimes Z_{-\chi}(\mu): \widehat{L}_0(\kappa)] = \sum_{\sigma,\tau\in\Lambda_I} [Z_\chi(\lambda):L_\chi(\sigma)][Z_{-\chi}(\mu):L_{-\chi}(\tau)][L_\chi(\sigma)\otimes L_{-\chi}(\tau):\widehat{L}_{0}(\kappa)].$$

\begin{lemma}\label{lem: twist comp mult}
There is equality $[Z_\chi(\lambda):L_\chi(\mu)]=[Z_{-\chi}(\lambda): L_{-\chi}(\mu)]$.
\end{lemma}

\begin{proof}

There exists $t\in T$ such that $\alpha(t)=-1$ for all $\alpha\in\Pi$ (see \cite[\S16.2]{Hu}). Since $\chi$ is in standard Levi form, the adjoint action of $t$ on $\g$ thus induces an isomorphism $\Phi_t:U_{-\chi}(\g)\xrightarrow{\sim} U_\chi(\g)$. Any $U_\chi(\g)$-module $M$ can therefore be equipped with the structure of a $U_{-\chi}(\g)$-module via $\Phi_t$; we denote the resulting $U_{-\chi}(\g)$-module by $M^t$. Since $\Phi_t$ is in fact an isomorphism of $X(T)/\Z I$-graded algebras, it further induces an equivalence of categories $\sC_\chi\to \sC_{-\chi}$. 

We claim that $$Z_\chi(\lambda)^t=Z_{-\chi}(\lambda);$$ since these modules have the same $\bK$-dimension it is enough to find a surjective homomorphism $Z_{-\chi}(\lambda)\to Z_{\chi}(\lambda)^t$ in $\sC_{-\chi}$. By Frobenius reciprocity, homomorphisms $Z_{-\chi}(\lambda)\to Z_{\chi}(\lambda)^t$ are in bijection with elements $v\in (Z_{\chi}(\lambda)^t)_{\lambda+\Z I} $ with the properties (1) that $h\cdot v=d\lambda(h) v$ for all $h\in\t$ and (2) that $\n^{+}\cdot v=0$. Consider the element $z_\lambda=1\otimes 1 \in Z_\chi(\lambda)_{\lambda+\Z I}$, which we view as lying in $(Z_\chi(\lambda)^t)_{\lambda+ \Z I}$. Then $h\cdot z_\lambda=(t\cdot h)z_\lambda=hz_\lambda=d\lambda(h) z_\lambda$ and $e_\beta\cdot z_\lambda = (t\cdot e_\beta)z_\lambda=\beta(t) e_{\beta}z_\lambda=0$ for all $\beta\in \Phi^{+}$. Thus $z_\lambda$ corresponds to a homomorphism $Z_{-\chi}(\lambda)\to Z_{\chi}(\lambda)^t$, which is surjective because $z_\lambda$ generates $Z_{\chi}(\lambda)^t$ as a $U_{-\chi}(\g)$-module.

Since $L_{-\chi}(\mu)$ is the unique simple quotient of $Z_{-\chi}(\mu)$ and since $L_\chi(\mu)^t$ is a simple quotient of $Z_\chi(\mu)^t$, we also have $L_{-\chi}(\mu)=L_\chi(\mu)^t$. The exactness of the equivalence of categories $\sC_\chi\to\sC_{-\chi}$ induced by $\Phi_t$ thus gives the result.

\end{proof}

As a consequence, we have 
$$[Z_\chi(\lambda)\otimes Z_{-\chi}(\mu): \widehat{L}_0(\kappa)] = \sum_{\sigma,\tau\in\Lambda_I} [Z_\chi(\lambda):L_\chi(\sigma)][Z_{\chi}(\mu):L_{\chi}(\tau)][L_\chi(\sigma)\otimes L_{-\chi}(\tau):\widehat{L}_{0}(\kappa)]$$ for all $\lambda,\mu,\kappa\in X(T)$.
Combining this with (\ref{e: ZoZ Comp1}) gives
\begin{multline*}
\sum_{0\leq a_i <p}\sum_{\gamma\in  \Z I} [Z_0(\lambda+\mu -a_D \gamma_D -\cdots - a_1\gamma_1): L_0(\kappa+p\gamma)]\\= \sum_{\sigma,\tau\in\Lambda_I} [Z_\chi(\lambda):L_\chi(\sigma)][Z_{\chi}(\mu):L_{\chi}(\tau)][L_\chi(\sigma)\otimes L_{-\chi}(\tau):\widehat{L}_{0}(\kappa)]
\end{multline*}
for all $\lambda,\mu,\kappa\in X(T)$.

\section{Minimal-dimensional modules}\label{s: Min Dim'l Mods}

We turn now to the second half of this paper, in which we consider minimal-dimensional $U_\chi(\g)$-modules. We begin in Subsection~\ref{ss: Prelim 2} by re-establishing some notation for this section of the paper. Subsection~\ref{ss: Rep Theory} then introduces some of the representation-theoretic tools we use in this section, largely following \cite{Pr07} for the constructions, and Subsection~\ref{ss: Assoc Var} develops some of the geometric tools (again, largely following \cite{Pr07}). The substance of this section is then in Subsection~\ref{ss: Type A}, where we focus on the case of $\GL_N$ and show (under some assumptions) that each minimal-dimensional $U_\chi(\gl_N(\bK))$-module arises as a quotient modules of certain modules obtained by base change of simple highest weight modules from characteristic zero.

\subsection{Preliminaries (Reprise)}\label{ss: Prelim 2}

In this subsection, we introduce the notation that is to be used in this second half of the paper. Many of the concepts will be familiar from Section~\ref{s: Tensor Prods}; however, as we need to work over various different fields and rings in this section we emphasise these more heavily in the notation. As such, with the exception of those concepts which only make sense over fields of positive characteristic (like the $p$-centre or reduced enveloping algebras), we use this subsection to reset our notation. Any notation omitted in this subsection can safely be taken to mean the same thing as in Section~\ref{s: Tensor Prods}.

In this section, let $G_{\C}$ be either a simple simply connected complex algebraic group or $\GL_N(\C)$, with Lie algebra $\g_\C$ either a complex simple Lie algebra or $\gl_N(\C)$. Fix a maximal torus $T_{\C}$ of $G_{\C}$ and a (positive) Borel subgroup $B_\C$ of $G_\C$ containing $T_\C$; the respective Lie algebras $\t_\C$ and $\b_\C$ are a Cartan subalgebra and a Borel subalgebra of $\g_\C$, respectively. The root system corresponding to $(G_\C,T_\C)$ is denoted $\Phi$, the system of positive roots corresponding to $B_\C$ is denoted $\Phi^{+}$, and the associated set of simple roots is denoted $\Pi$. Given $\alpha\in\Phi$, we set $\g_{\C,\alpha}=\{x\in\g\mid [h,x]=\alpha(h)x \,\,\mbox{for all}\,\, h\in\t_\C\}$ and we set $$\n_\C^{+}=\bigoplus_{\alpha\in\Phi^{+}}\g_{\C,\alpha}\quad\quad\mbox{and}\quad \quad\n_\C^{-}=\bigoplus_{\alpha\in\Phi^{+}}\g_{\C,-\alpha}.$$ We then have $\b_{\C}=\t_\C\oplus \n^{+}_\C$ and $\g_\C=\n_\C^{-}\oplus \t_\C \oplus \n_\C^{+}$. Let $N_{G_{\C}}(T_\C)$ be the normaliser of $T_\C$ in $G_\C$; the Weyl group of $(G_\C,T_\C)$ is then $W=N_{G_\C}(T_\C)/T_\C$, which is a finite group. Given $w\in W$, we denote by $\dot{w}$ an (arbitrary) lift of $w$ to an element of $N_{G_\C}(T_\C)\subseteq G_\C$. In the case when $G_\C=\GL_N(\C)$, the Weyl group can be identified with $S_N$. In this case, we always choose $\dot{w}$ to be a permutation matrix, which therefore determines it uniquely.

When $G_\C$ is simple, we take $$\cB=\{e_\alpha\mid \alpha\in\Phi\} \cup \{h_\beta\mid \beta\in \Pi\}$$ to be a Chevalley basis of $\g_\C$, where $e_\alpha\in\g_{\C,\alpha}$ and $h_\beta=[e_\beta,e_{-\beta}]\in \t_\C$. When $G_\C=\GL_N(\C)$ we take $$\cB=\{e_{ij} \mid 1\leq i,j\leq N\},$$ where $e_{ij}$ is the $N\times N$ matrix with a $1$ in the $(i,j)$-th position and zeroes in all other positions. In either case, we equip $\cB$ with a total ordering; at present, we do so arbitrarily.

We also define the following subsets of $\cB$ when $G_\C$ is simple (resp. when $G_\C=\GL_N(\C)$):
$$\cB^{\pm}=\{ e_{\alpha}\mid\alpha\in\Phi^{\pm}\},\quad \left(\mbox{resp. } \cB^{\pm}=\{e_{ij}\mid \pm(i-j)<0\} \right),$$
$$\cB^{0}=\{h_\beta\mid \beta\in \Pi\},\quad \left(\mbox{resp. } \cB^0=\{e_{ii} \mid 1\leq i\leq N\} \right),$$
$$\cB^{\geq}=\cB^{0}\cup \cB^{+}.$$

Let $R=\cS^{-1}\Z$ be the localisation of $\Z$ at a finitely-generated multiplicatively-closed subset $\cS$ of $\Z$. At present we assume only that $\cS$ must contain all primes which are not very good\footnote{In particular, if $\Phi$ has type $A_{N-1}$ then $S$ must contain all $p\mid N$ -- even if $G_\C=\GL_N(\C)$.} for $\Phi$ (although for some results later on we must impose further assumptions on $\cS$). We call the assumption that $\cS$ (and thus $R$) satisfies these properties ``Assumption (R1)'', and use similar labelling for future assumptions we require. We may then define an $R$-form of $\g_\C$ by $$\g_R=R\mbox{-span}(\cB);$$ by definition, we have $\g_\C\cong \g_R\otimes_R \C$. 

Let $\bK$ be an algebraically closed field of positive characteristic $p>0$. Throughout the remainder of the paper, we make the following assumption:
$$p\mbox{ \bf{is invertible in} }R.$$
In particular, as we impose further restrictions on the properties that $R$ must satisfy, we consequently limit which $p$ may be considered. We define $\g_\bK=\g_R\otimes_{R} \bK$; this is the Lie algebra of the algebraic group $G_\bK$ over $\bK$ with the same root datum as $G_\C$. Under our assumptions, $G_\bK$ satisfies the standard hypotheses (A), (B) and (C). Furthermore, when $G_\C$ is simple our assumptions on $p$ are enough to guarantee that $\g_{\bK}$ is simple.

We similarly define $\t_R=R\mbox{-span}(\cB^0)$, $\b_R=R\mbox{-span}(\cB^{\geq})$, $\n_R^{+}=R\mbox{-span}(\cB^{+})$ and $\n_R^{-}=R\mbox{-span}(\cB^{-})$, and set $\t_\bK=\t_R\otimes_R \bK$, $\b_\bK=\b_R\otimes_R \bK$, $\n_\bK^{+}=\n_R^{+}\otimes_R \bK$ and $\n_\bK^{-}=\n_R^{-}\otimes_R \bK$; we identify these latter vector spaces with subspaces of $\g_\bK$. We define $\t_\C^{*}$ to be the $\C$-vector space of linear maps $\t_\C\to\C$, $\t_R^{*}$ to be the free $R$-module of $R$-linear maps $\t_R\to R$, and $\t_\bK^{*}$ to be the vector space of $\bK$-linear maps $\t_\bK\to\bK$. Note that our choice of basis $\cB^0$ gives an embedding $\t_R^{*}\hookrightarrow \t_\C^{*}$ and a homomorphism $\t_R^{*}\to \t_{\bK}^{*}$ with kernel $p\t_R^{*}$; given $\lambda\in \t_R^{*}$, we generally write $\widetilde{\lambda}$ for the image of $\lambda$ under the latter map. The only exception to this will be for roots -- given a root $\alpha\in\Phi$, we continue to denote it by $\alpha$ whether we view it as an element of $\t_\C^{*}$, $\t_R^{*}$ or $\t_\bK^{*}$.\footnote{In fact, consistent with Section~\ref{s: Tensor Prods}, we also use this notation for the corresponding element of the character groups $X(T_\C)$ and $X(T_\bK)$.}

In the same way, we write $\g_\C^{*}$, $\g_R^{*}$ and $\g_\bK^{*}$ for the vector space/free module of linear maps $\g_\C\to\C$, $\g_R\to R$ and $\g_\bK\to\bK$ respectively. As with the Cartan subalgebras, our choice of basis $\cB$ gives an inclusion $\g_R^{*}\hookrightarrow \g_\C^{*}$ and a homomorphism $\g_R^{*}\to \g_\bK^{*}$ with kernel $p\g_R^{*}$. Noting that $\g_\bK^{*}=\g_R^{*}\otimes_R \bK$, the latter map identifies with the map $\chi\mapsto \chi\otimes 1$. In this setting, we tend to abuse notation and write $\chi$ for the corresponding element in each of $\g_\C^{*}$, $\g_R^{*}$ and $\g_\bK^{*}$; when it may cause confusion, however, we occasionally write $\chi\otimes 1$ for the image of $\chi\in\g_R^{*}$ in $\g_\bK^{*}$.

\subsection{Representation theory}\label{ss: Rep Theory}

We introduce here some of the representation-theoretic objects we use further below. The constructions here largely follow \cite{Pr07}.

Given $\lambda\in\t_\C^{*}$, the Verma module $M_\C(\lambda)$ corresponding to $\lambda$ is the $U(\g_\C)$-module defined as $M_{\C}(\lambda):=U(\g_\C)\otimes_{U(\b_\C)}\C_\lambda$, where $\C_\lambda$ is the one-dimensional $U(\b_\C)$-module on which $\n^{+}_\C$ acts trivially and $\t_\C$ acts via $\lambda$. This has a unique simple quotient $L_\C(\lambda)$, which is finite-dimensional if and only if $\lambda$ is integral and dominant, i.e. $\lambda(h_\alpha)\in\Z_{\geq 0}$ for all $\alpha\in \Pi$. Let us denote $v_\lambda:=1\otimes 1\in M_\C(\lambda)$, a highest weight vector in $M_\C(\lambda)$, and let us denote by $\overline{v}_\lambda$ the image of $v_\lambda$ in $L_\C(\lambda)$. We further denote by $M^{\max}_\C(\lambda)$ the unique maximal submodule of $M_\C(\lambda)$, so that $M_\C(\lambda)/M^{\max}_\C(\lambda)=L_\C(\lambda)$.

The universal enveloping algebra $U(\g_\C)$ has a Poincar\'{e}-Birkhoff-Witt $\C$-basis consisting of ordered monomials in $\cB$ (with respect to the chosen total order on $\cB$). Note further that $\g_R$ is a Lie ring and has a universal enveloping $R$-algebra $U(\g_R)$. Moreover, $U(\g_R)$ is a free $R$-algebra which also has an $R$-basis of ordered monomials in $\cB$. In particular, this implies that $U(\g_\C)\cong U(\g_R)\otimes_R \C$. 

Given $\lambda\in \t_R^{*}$, define $M_R(\lambda)=U(\g_R)v_\lambda$, which is an $R$-form of $M_\C(\lambda)$ by construction. We then define the $U(\g_R)$-modules $M_R^{\max}(\lambda):=M_R(\lambda)\cap M^{\max}(\lambda)$ and $L_R(\lambda):=M_R(\lambda)/M_R^{\max}(\lambda)$. We may alternately define $L_R(\lambda)=U(\g_R)\overline{v}_\lambda$; it is straightforward to see that these two definitions coincide. The important point for present purposes is that there is a surjective homomorphism of $U(\g_R)$-modules $M_R(\lambda)\twoheadrightarrow L_R(\lambda)$.

Note now that $U(\g_\bK)$ also has a Poincar\'{e}-Birkhoff-Witt basis of monomials in $\cB$, and we thus have $U(\g_\bK)\cong U(\g_R)\otimes_R \bK$. In particular we may define the $U(\g_\bK)$-modules $M_p(\lambda):=M_R(\lambda)\otimes_R \bK$ and $L_p(\lambda):=L_R(\lambda)\otimes_R \bK$. It is straightforward to check that we may alternatively define $M_p(\lambda)=U(\g_\bK)\otimes_{U(\b_{\bK})} \bK_\lambda$, where $\bK_\lambda$ is the one-dimensional $U(\b_\bK)$-module on which $\n^{+}_{\bK}$ acts via zero and $\t_{\bK}$ acts via $\widetilde{\lambda}$. Letting $w_\lambda:=v_\lambda\otimes 1\in M_p(\lambda)$ and $\overline{w}_\lambda:=\overline{v}_\lambda \otimes 1\in L_p(\lambda)$, the surjection $M_R(\lambda)\twoheadrightarrow L_R(\lambda)$ induces a surjective homomorphism of $U(\g_\bK)$-modules $M_p(\lambda)\twoheadrightarrow L_p(\lambda)$ sending $w_\lambda$ to $\overline{w}_\lambda$.

Given $\chi\in \g_\bK^{*}$, define $J_\chi:=\langle x^{p}-x^{[p]}-\chi(x)^p\mid x\in\g_\bK\rangle\subseteq Z_p(\g_\bK)$. Then we may form the submodule $J_\chi M_p(\lambda)\subseteq M_p(\lambda)$ and the quotient module $M_p^\chi(\lambda):=M_p(\lambda)/J_\chi M_p(\lambda)$. By the same token, $J_\chi L_p(\lambda)$ is a $U(\g_\bK)$-submodule of $L_p(\lambda)$ and we may form the quotient module $L_p^\chi(\lambda):=L_p(\lambda)/J_\chi L_p(\lambda)$. Both $M_p^\chi(\lambda)$ and $L_p^\chi(\lambda)$ are then (possibly trivial) $U_\chi(\g_{\bK})$-modules, and the surjection $M_p(\lambda)\twoheadrightarrow L_p(\lambda)$ induces a surjection $M_p^\chi(\lambda)\twoheadrightarrow L_p^\chi(\lambda)$ of $U_\chi(\g_{\bK})$-modules.

The next lemma shows that the $U_\chi(\g_\bK)$-modules $M_p^\chi(\lambda)$ in fact coincide with certain modules we discussed extensively in Section~\ref{s: Tensor Prods} of this paper: baby Verma modules.

\begin{prop}\label{prop: M and Z iso}
The $U_\chi(\g_{\bK})$-modules $M_p^\chi(\lambda)$ and $Z_\chi(\widetilde{\lambda})$ are isomorphic. 
\end{prop}

\begin{proof}
Since $\n_{\bK}^{+}w_\lambda=0$ and $hw_\lambda=\widetilde{\lambda}(h)w_\lambda$ for all $h\in\t_{\bK}$, Frobenius reciprocity for $Z_\chi(\widetilde{\lambda})$ implies that there exists a homomorphism $Z_\chi(\widetilde{\lambda})\to M_p^\chi(\lambda)$. Similarly, as $\n_{\bK}^{+}z_\lambda=0$ and $hz_\lambda=\widetilde{\lambda}(h)z_\lambda$, Frobenius reciprocity for $M_p(\lambda)$ induces a homomorphism of $U(\g_{\bK})$-modules $M_p(\lambda)\to Z_\chi(\widetilde{\lambda})$. Since $J_\chi Z_\chi(\widetilde{\lambda})=0$, this induces a homomorphism $M_p^\chi(\lambda)\to Z_\chi(\widetilde{\lambda})$. It is straightforward to see that the maps $Z_\chi(\widetilde{\lambda})\to M_p^\chi(\lambda)$ and $M_p^\chi(\lambda)\to Z_\chi(\widetilde{\lambda})$ are inverse to each other.
\end{proof}

This proposition allows us to prove the following corollary. To state it, recall that when $\chi$ is in (weak) standard Levi form each baby Verma module $Z_\chi(\widetilde{\lambda})$ has a unique simple quotient which we denote $L_\chi(\widetilde{\lambda})$.

\begin{cor}\label{cor: surj hom simple}
Suppose that $\chi$ is in standard Levi form and that $L_p^\chi(\lambda)\neq 0$. Then there is a surjective homomorphism of $U_\chi(\g_{\bK})$-modules $L_p^\chi(\lambda)\twoheadrightarrow L_\chi(\widetilde{\lambda})$.
\end{cor}

\begin{proof}
By Proposition~\ref{prop: M and Z iso}, $M_p^\chi(\lambda)=Z_\chi(\widetilde{\lambda})$. Therefore, it has unique simple quotient $L_\chi(\widetilde{\lambda})$. We have already observed that $M_p^\chi(\lambda)$ surjects onto $L_p^\chi(\lambda)$ (and thus, in particular, $L_p^\chi(\lambda)$ is finite-dimensional). If $L_p^\chi(\lambda)\neq 0$ then it must have a simple quotient, which must also be a simple quotient of  $M_p^\chi(\lambda)$. The result follows.
\end{proof}

Consider now only the case where $G_\C=\GL_N(\C)$. Then $G_\C':=\SL_N(\C)$ is a simple simply connected complex algebraic group, with Lie algebra $\g_\C'=\sl_N(\C)$. Writing $\t_\C'=\t_\C\cap \sl_N(\C)$, there is a natural surjection $\t_\C^{*}\twoheadrightarrow (\t_\C')^{*}$ which we write as $\lambda\mapsto \lambda'$. One can easily see that restricting the simple $U(\gl_N(\C))$-module $L_\C(\lambda)$ to $U(\sl_N(\C))$ gives the simple $U(\sl_N(\C))$-module $L_\C(\lambda')$. Furthermore, it is a straightforward exercise to check that, based on our choices of bases $\cB$, restriction of the $U(\gl_N(\bK))$-module $L_p(\lambda)$ to $U(\sl_N(\bK))$ gives precisely the module $L_p(\lambda')$ which we obtain by directly applying in the setting of $G_\C=\SL_N(\C)$ the constructions of Subsection~\ref{ss: Prelim 2} and this subsection.

As with the Cartan subalgebras, we may restrict an element $\chi\in\gl_N(\bK)^{*}$ to an element $\chi'\in \sl_N(\bK)^{*}$. Define then $J_\chi=\langle x^p-x^{[p]}-\chi(x)\mid x\in\gl_N(\bK)\rangle\subseteq Z_p(\gl_N(\bK))$ and $J'_{\chi'}=\langle x^p-x^{[p]}-\chi(x)\mid x\in\sl_N(\bK)\rangle\subseteq Z_p(\sl_N(\bK))$, and consider the $U_\chi(\gl_N(\bK))$-module $L_p^\chi(\lambda)=L_p(\lambda)/J_\chi L_p(\lambda)$ and the $U_{\chi'}(\sl_N(\bK))$-module $L_p^{\chi'}(\lambda')=L_p(\lambda')/J_{\chi'}' L_p(\lambda')$. These two modules are then related as follows -- note that $U_{\chi'}(\sl_N(\bK))$ is a subalgebra of $U_\chi(\gl_N(\bK))$.

\begin{lemma}
Let $y\in\gl_N(\bK)$ denote the identity matrix. If $\chi(y)=0$ then restriction of the $U_\chi(\gl_N(\bK))$-module $L_p^\chi(\lambda)$ to $U_{\chi'}(\sl_N(\bK))$ precisely gives the $U_{\chi'}(\sl_N(\bK))$-module $L_p^{\chi'}(\lambda')$.
\end{lemma}

\begin{proof}
Since $L_p(\lambda)$ and $L_p(\lambda')$ coincide as $U(\sl_N(\C))$-modules, it suffices to show that there is equality $J_\chi L_p(\lambda)=J_{\chi'}' L_p(\lambda')$. 
The $p$-centre $Z_p(\gl_N(\C))$ coincides with the polynomial subalgebra $\bK[e_{11}^p-e_{11},\ldots, e_{NN}^p-e_{NN}]$ of $U(\gl_N(\C))$. Setting $x_{ii}=e_{i-1,i-1}-e_{i,i}$ for $i=2,\ldots,N$, and recalling that $y=e_{11}+\cdots+e_{NN}$, we may rewrite this subalgebra as the polynomial algebra $\bK[x_{22}^p-x_{22},\ldots, x_{NN}^p-x_{NN},y^p-y]$ so long as $p\nmid N$ (which for us is imposed by our assumptions on $R$). Since $Z_p(\sl_N(\bK))=\bK[x_{22}^p-x_{22},\ldots, x_{NN}^p-x_{NN}]$, we may rewrite this as $Z_p(\gl_N(\bK))=Z_p(\sl_N(\bK))+\langle y^p-y\rangle$.

From this perspective, it is straightforward to see that $J_\chi=J_{\chi'}'+ \langle y^p\rangle$. Since $y$ acts on $L_p(\lambda)$ via scalar multiplication by $\widetilde{\lambda}(y)\in R$ (which follows from the easy-to-see analogous statement for $M_\C(\lambda)$ and $M_p(\lambda)$), we get that $(y^p-y)v=(\widetilde{\lambda}(y)^p-\widetilde{\lambda}(y))v=0$ for all $v\in L_p(\lambda)$. Thus, $J_\chi L_p(\lambda)=J_{\chi'}' L_p(\lambda')$, as required.

\end{proof}

By definition, $\chi\in\gl_N(\C)^{*}$ is in standard Levi form if and only if $\chi\in\sl_N(\C)^{*}$ is in standard Levi form. 

\begin{cor}\label{cor: gl v sl}
Suppose that $\chi\in\gl_N(\C)^{*}$ is in standard Levi form. Then $L_p^\chi(\lambda)\neq 0$ if and only if $L_p^{\chi'}(\lambda')\neq 0$.
\end{cor}

In particular, this corollary means that to determine when $L_p^\chi(\lambda)\neq 0$ for $G_\C=\GL_N(\C)$ (and thus to apply Corollary~\ref{cor: surj hom simple}), we need only consider the situation over $\SL_N(\C)$. This will make some exposition easier, and therefore in the next subsection  we consider only $G_\C$ simple.

\subsection{Associated Varieties}\label{ss: Assoc Var}

In light of Corollary~\ref{cor: surj hom simple}, we want to determine when $L_p^\chi(\lambda)$ is non-zero. The main tool to do so is \cite[Lemma 3.1]{Pr07}; to explain this result, we first need to introduce various subvarieties of $\g_\C^{*}$ and $\g_\bK^{*}$. We proceed to do so in this subsection. As mentioned above, we (largely) only consider $G_\C$ which are simple simply connected complex algebraic groups here, for ease of exposition. Our explanations follow \cite{JanNO,Pr07}.

As we have already observed, the universal enveloping algebra $U(\g_{\C})$ has a Poincar\'{e}-Birkhoff-Witt basis, which induces a PBW-filtration on $U(\g_{\C})$. More specifically, if we temporarily write the basis of $\g_{\C}$ as $x_1,\ldots, x_k$, then $U(\g_{\C})$ has a basis consisting of the elements $$x_1^{a_1}\cdots x_k^{a_k}\qquad \mbox{for} \qquad a_1,\ldots,a_k\geq 0.$$ The PBW-filtration of $U(\g_\C)$ is then given by $$\C=U_0(\g_{\C})\subseteq U_1(\g_{\C})\subseteq U_2(\g_{\C})\subseteq \cdots \subseteq U(\g_{\C})$$ where $$U_n(\g_{\C}):=\C\mbox{-span}\{x_1^{a_1}\cdots x_k^{a_k} \mid a_1+\cdots + a_k \leq n \}.$$

It is well-known that the associated graded algebra $\gr U(\g_{\C}):=\bigoplus_{n\geq 0}U_{n}(\g)/U_{n-1}(\g)$ of $U(\g_{\C})$ with respect to this filtration (setting $U_{-1}(\g_\C)=0$) is isomorphic to the symmetric algebra $S(\g_{\C})$. Given an ideal $I$ of $U(\g_{\C})$, we may form an ideal of $S(\g_{\C})$ as $$\gr I=\bigoplus_{n\geq 0} \left((I\cap U_{n}(\g_{\C}))+ U_{n-1}(\g_{\C})\right)/U_{n-1}(\g_{\C})\subseteq S(\g_{\C})\cong \C[\g_{\C}^{*}].$$ The {\bf associated variety} $\cV\cA(I)$ of $I$ is then defined to be the zero set of $\gr I$ in $\g_{\C}^{*}$. When $I$ is a primitive ideal (i.e. the annihilator of a simple $U(\g_{\C})$-module), Joseph's Irreducibility Theorem says that $$\cVA(I)=\overline{\O}$$ for some nilpotent coadjoint $G_{\C}$-orbit $\O\subseteq \g_{\C}^{*}$ (see \cite[Corollary 3.3]{Jo81} for type $A$, \cite[Theorem 3.10]{Jo85} for other types). Using the Killing form, we can identify $\g_{\C}$ and $\g_{\C}^{*}$ and thus identify $\cVA(I)=\overline{\O}$ with a subvariety of $\g_\C$ (we abuse notation and use $\cVA(I)$ to denote this subvariety as well).

\begin{rmk}
When $G_\C=\GL_N(\C)$ and $I$ is a primitive ideal in $U(\g_\C)$, we may apply the same construction to get the associated variety $\cVA(I)\subseteq \gl_N(\C)^{*}$. Using the trace form (instead of the Killing form, which is no longer non-degenerate), we may also identify this with a subvariety of $\gl_N(\C)$ which, again, we still denote by $\cVA(I)$. We then denote by $\cVA'(I)$ the image of $\cVA(I)$ under the surjection $\gl_N(\C)\twoheadrightarrow\pgl_N(\C)\xrightarrow{\sim} \sl_N(\C)$. Joseph's Irreducibility Theorem again gives that $\cVA'(I)$ is the closure of a nilpotent orbit. (Although we don't use it in this paper, a similar definition may be used whenever $G_\C$ is reductive.)  
\end{rmk}

Given $\lambda\in\t_R^{*}$, the $U(\g_{\C})$-module $L_{\C}(\lambda)$ is simple and thus the annihilator $I(\lambda):=\Ann_{U(\g_{\C})} L_{\C}(\lambda)$ is a primitive ideal of $U(\g_{\C})$. There thus exists a nilpotent coadjoint $G_{\C}$-orbit, which we henceforth denote $\O_\lambda$, such that $\cVA(I(\lambda))=\overline{\O}_\lambda$.

The PBW-filtration on $U(\g_{\C})$ induces a filtration on $L_{\C}(\lambda)$ by setting $L_{\C,n}(\lambda)=U_n(\g_{\C})\overline{v}_\lambda$. This filtration is compatible with the filtration of $U(\g_{\C})$ in the sense that $U_{n}(\g_{\C}) L_{\C,m}(\lambda)\subseteq L_{\C, m+n}(\lambda)$. From this, we may form the associated graded module $$\gr L_{\C}(\lambda):= \bigoplus_{n\geq 0} L_{\C,n}(\lambda)/L_{\C,n-1}(\lambda),$$ which is an $S(\g_{\C})$-module. Note that $\overline{v}_\lambda\in L_{\C,0}(\lambda)$; we abuse notation to write $\overline{v}_\lambda=\gr_0(\overline{v}_\lambda)\in \gr L_{\C}(\lambda)$. It remains true that $\overline{v}_\lambda$ generates $\gr L_{\C}(\lambda)$ as an $S(\g_{\C})$-module, i.e. $ \gr L_{\C}(\lambda)=S(\g_{\C})\overline{v}_\lambda $. Since $S(\g_{\C})$ is commutative, we have $$J(\lambda):=\Ann_{S(\g_{\C})} \gr L_{\C}(\lambda) = \Ann_{S(\g_{\C})} \overline{v}_\lambda \subseteq S(\g_{\C}).$$ The associated variety $\cV_{\g_{\C}} L_{\C}(\lambda)$ is then defined to be the zero set of $J(\lambda)$ in $\g_{\C}^{*}$, i.e. the set of maximal ideals in $S(\g_{\C})$ containing $J(\lambda)$.

The universal enveloping $R$-algebra $U(\g_R)$ corresponding to the Lie ring $\g_R$ also has a PBW-filtration
$$R=U_0(\g_R)\subseteq U_1(\g_R)\subseteq U_2(\g_R)\subseteq \cdots \subseteq U(\g_R),$$ with associated graded algebra isomorphic to $S(\g_R)$. Using this filtration, we may define a corresponding filtration of $L_R(\lambda)$ by setting $$L_{R,n}(\lambda)=U_n(\g_R)\overline{v}_\lambda.$$ We form as before the associated graded module $\gr L_{R}(\lambda)$ and the annihilator $$J_R(\lambda):=\Ann_{S(\g_R)} \gr L_R(\lambda)=\Ann_{S(\g_R)} \overline{v}_\lambda.$$

We may do similar over $\bK$: as before, $U(\g_{\bK})$ has a PBW basis and a PBW-filtration, allowing us define a filtration $$0=L_{p,-1}(\lambda)\subseteq L_{p,0}(\lambda)\subseteq L_{p,1}(\lambda)\subseteq \cdots \subseteq L_{p}(\lambda)$$ of $L_{p}(\lambda)$. We may define the associated graded module $\gr L_p(\lambda):=\bigoplus_{n\geq 0}L_{p,n}(\lambda)/L_{p,n-1}(\lambda)$, an $S(\g_{\bK})$-module. We then define $$J_p(\lambda):=\Ann_{S(\g_{\bK})} \gr L_p(\lambda) = \Ann_{S(\g_{\bK})} \overline{w}_\lambda\subseteq S(\g_{\bK}).$$ The associated variety $\cV_{\g_{\bK}} L_p(\lambda)\subseteq \g_{\bK}^{*}$ of $L_p(\lambda)$ is defined to be the zero set of $J_p(\lambda)$, i.e. the set of maximal ideals of $ S(\g_\bK) $ containing $J_p(\lambda)$.

Finally, since $L_p(\lambda)$ is a $U(\g_{\bK})$-module it restricts to a module over the $p$-centre $Z_p(\g_{\bK})$. We then define $$\cI_p(\lambda):=\Ann_{Z_{p}(\g_{\bK})} L_p(\lambda)=\Ann_{Z_{p}(\g_{\bK})} \overline{w}_\lambda \subseteq Z_p(\g_{\bK}).$$ The $p$-centre $Z_p(\g_\bK)$ identifies with $S(\g_\bK^{(1)})$, where $\cdot\,^{(1)}$ denotes the Frobenius twist.\footnote{More explicitly, $\g_\bK^{(1)}$ is precisely the Lie algebra $\g_\bK$ except with scalar multiplication now given by $a\cdot x=a^{1/p}x$.} Since the Frobenius morphism on $\bK$ is bijective, there is a $\bK$-algebra isomorphism $S(\g_\bK^{(1)})\cong S(\g_\bK)$; the maximal spectrum of $Z_p(\g_\bK)$ can thus be identified with $\g_{\bK}^{*}$, and hence the vanishing set of $\cI_p(\lambda)$ can be identified with a subvariety of $\g_\bK^{*}$. We denote the latter variety by $V_p(\lambda)\subseteq \g_\bK^{*}$.

As discussed earlier, the following result of \cite[Lemma 3.1]{Pr07} is important for what follows.
\begin{lemma}\label{lem: Vp Lchi n0}
Let $\chi\in\g_\bK^{*}$. If $\chi\in V_p(\lambda)$ then $L_p^\chi(\lambda)\neq 0$.
\end{lemma}

In light of Corollary~\ref{cor: surj hom simple}, we therefore wish to understand when $\chi\in\g_\bK^{*}$ in standard Levi form lies in some $V_p(\lambda)$. In the remainder of this subsection, we gather some results to address this question.

In order for these results to apply, however, we need some further assumptions on $R=\cS^{-1}\Z$ (and therefore some further restrictions on $p>0$). First, {\bf we assume that $R$ is such that $S(\g_R)/ J_R(\lambda)$ has no $R$-torsion}. We call this ``Assumption (R2($\lambda$))'', noting that it depends on $\lambda\in\t_\Q^{*}$ (and therefore part of the assumption is that $\lambda$ may be defined as an element of $\t_R^{*}$). We may always make Assumption (R2($\lambda$)) for a given $\lambda\in\t_\Q^{*}$ by extending $\cS$ if necessary (see \cite[\S 2.2]{Pr07}). Under this assumption each $L_{R,n}(\lambda)/L_{R,n-1}(\lambda)$ is torsion-free as an $R$-module and therefore is a free $R$-module of finite rank.

\begin{prop}\label{prop: R to K in VL}
Let $\lambda\in\t_R^{*}$ and make assumption (R2($\lambda)$)). If $\chi:\g_R\to R$ is an $R$-linear map such that $\chi\in\cV_{\g_{\C}} L_{\C}(\lambda)$ (when $\chi$ is viewed as an element of $\g_\C^{*}$), then $\chi\otimes 1\in\g_{\bK}^{*}$ lies in $\cV_{\g_{\bK}} L_{p}(\lambda)$.
\end{prop}

\begin{proof}
For this proof, as above, let us enumerate the Chevalley basis $\cB$ of $\g_\C$ as $x_1,\ldots,x_k$; by abuse of notation, we also use this convention for analogous bases of $\g_R$ and $\g_\bK$. In particular, the algebras $S(\g_\C)$, $S(\g_R)$ and $S(\g_\bK)$ all have bases consisting of elements  $$x_1^{a_1}\cdots x_k^{a_k}\qquad \mbox{for} \qquad a_1,\ldots,a_k\geq 0.$$

To prove the result, we first prove the existence of the following chain of isomorphisms. Each of the vertical maps either sends a monomial $x_1^{a_1}\cdots x_k^{a_k}$ to the monomial $x_1^{a_1}\cdots x_k^{a_k}$ (possibly in a different algebra) or sends a monomial $x_1^{a_1}\cdots x_k^{a_k}$ to its image under the action map (so either $x_1^{a_1}\cdots x_k^{a_k}(\overline{v}_\lambda\otimes 1)$ or $(x_1^{a_1}\cdots x_k^{a_k}\overline{v}_\lambda)\otimes 1$). The horizontal map is the obvious one, but will be discussed more extensively later. In this diagram, $\kappa$ is the natural isomorphism $S(\g_R)\otimes_R \bK\to S(\g_\bK)$ and $\iota$ is the map $\Ann_{S(\g_R)}(\overline{v}_\lambda)\otimes_R \bK\to S(\g_R)\otimes_R \bK$ induced from the inclusion $\Ann_{S(\g_R)}(\overline{v}_\lambda) \hookrightarrow S(\g_R)$.
\begin{eqnarray}
	\label{e: Sym alg isom}
	\begin{array}{c}\xymatrix{
			\frac{S(\g_\bK)}{\Ann_{S(\g_\bK)}(\overline{w}_\lambda)}  \ar@{<-}[d]^{(1)} & & \frac{S(\g_\bK)}{\kappa\iota(\Ann_{S(\g_R)}(\overline{v}_\lambda)\otimes_R \bK)} \ar@{<-}[d]^{(6)} \\
			\frac{S(\g_R)\otimes_R\bK}{\Ann_{S(\g_R)\otimes_R \bK}(\overline{v}_\lambda\otimes 1)} \ar@{->}[dd]^{(2)} & & \frac{S(\g_R)\otimes_R\bK}{\iota(\Ann_{S(\g_R)}(\overline{v}_\lambda)\otimes_R \bK)} \ar@{->}[d]^{(5)}
			\\ 
			& &  \frac{S(\g_R)}{\Ann_{S(\g_R)}(\overline{v}_\lambda)}\otimes_R \bK \ar@{->}[d]^{(4)}
			\\
			\gr(L_R(\lambda)\otimes_R \bK) \ar@{->}[rr]^{(3)} & & (\gr L_R(\lambda))\otimes_R \bK 
		}
	\end{array}
\end{eqnarray}

That the maps labelled (1) and (6) are isomorphisms follows easily from the fact that $\kappa$ is an isomorphism and clearly sends $\Ann_{S(\g_R)\otimes_R \bK}(\overline{v}_\lambda\otimes 1)$ to $\Ann_{S(\g_\bK)}(\overline{w}_\lambda)$ (here $\overline{v}_\lambda\otimes 1$ is an element of $\gr(L_R(\lambda)\otimes_R \bK)$, not an element of $\gr(L_R(\lambda))\otimes_R \bK$) and sends $\iota(\Ann_{S(\g_R)}(\overline{v}_\lambda)\otimes_R \bK)$ to $\kappa\iota(\Ann_{S(\g_R)}(\overline{v}_\lambda)\otimes_R \bK)$. Furthermore, that the maps labelled (2) and (4) are isomorphisms follows immediately from the description of these maps and the first isomorphism theorem (recalling that $\gr(L_R(\lambda)\otimes_R \bK)$ and $\gr(L_R(\lambda))$ are generated by $\overline{v}_\lambda\otimes 1$ and $\overline{v}_\lambda$ respectively). And the map (5) is an isomorphism due to the right-exactness of $-\otimes_R\bK$ and the exactness of the sequence of $R$-modules $$0\to \Ann_{S(\g_R)}(\overline{v}_\lambda)\to S(\g_R)\to S(\g_R)/\Ann_{S(\g_R)}(\overline{v}_\lambda)\to 0.$$

What remains is therefore to show that the map labelled (3) is an isomorphism. As in \cite{Pr07}, this follows from our assumption that each $L_{R,n}(\lambda)/L_{R,n-1}(\lambda)$ is torsion-free over $R$, and thus free of finite rank. Indeed, this assumption is enough to show that $L_{p,n}(\lambda)\cong L_{R,n}(\lambda)\otimes_R \bK$ and that the inclusion $L_{R,n-1}(\lambda)\hookrightarrow L_{R,n}(\lambda)$ induces an injection $L_{R,n-1}(\lambda)\otimes_R \bK\to L_{R,n}(\lambda)\otimes_R \bK$, which together imply the result. The map obtained via this argument sends $$x_1^{a_1}\cdots x_k^{a_k}\overline{v}_\lambda\otimes 1\mapsto x_1^{a_1}\cdots x_k^{a_k}\overline{v}_\lambda \otimes 1.$$

Putting all this together yields an isomorphism $\frac{S(\g_\bK)}{\Ann_{S(\g_\bK)}(\overline{w}_\lambda)}  \xrightarrow{\sim } \frac{S(\g_\bK)}{\kappa\iota(\Ann_{S(\g_R)}(\overline{v}_\lambda)\otimes_R \bK)}$ induced from the identity map $S(\g_\bK)\to S(\g_\bK)$. Therefore, $$\Ann_{S(\g_{\bK})}(\overline{w}_\lambda)=\kappa\iota(\Ann_{S(\g_R)}(\overline{v}_\lambda)\otimes_R \bK).$$

Now, given an $R$-linear map $\chi:\g_R\to R$, we define $$\m_\chi^{\C}=\langle x-\chi(x) \mid x\in \g_{\C}\rangle \subseteq S(\g_{\C}),$$
$$\m_\chi^{R}=\langle x-\chi(x) \mid x\in \g_R\rangle \subseteq S(\g_{R}), \quad \mbox{and} $$
$$\m_\chi^{\bK}=\langle x-(\chi\otimes 1)(x) \mid x\in \g_{\bK}\rangle \subseteq S(\g_{\bK}).$$ The first and last of these are maximal ideals, which correspond to the points $\chi\in\g_{\C}^{*}$ and $\chi\otimes 1\in \g_{\bK}^{*}$. The statement that $\chi\in\cV_{\g_\C} L_{\C}(\lambda)$ is therefore equivalent to the statement that $$\Ann_{S(\g_{\C})}(\overline{v}_\lambda)\subseteq \m_\chi^{\C}.$$

This then implies that $$\Ann_{S(\g_R)}(\overline{v}_\lambda)=\Ann_{S(\g_{\C})}(\overline{v}_\lambda)\cap S(\g_R)\subseteq \m_\chi^{\C}\cap S(\g_R)=\m_\chi^R,$$ and therefore that $$\iota(\Ann_{S(\g_R)}(\overline{v}_\lambda)\otimes_R \bK) \subseteq   \overline{\iota}(\m_\chi^R\otimes_R \bK),$$ where $\overline{\iota}$ is induced from the inclusion $\m_\chi^R\subseteq S(\g_R)$. Applying $\kappa$ to both sides gives $$\Ann_{S(\g_\bK)}(\overline{w}_\lambda)=\kappa\iota(\Ann_{S(\g_R)}(\overline{v}_\lambda)\otimes_R \bK) \subseteq   \kappa\overline{\iota}(\m_\chi^{R}\otimes_R \bK)=\m_\chi^\bK.$$ This precisely means that $\chi\otimes 1\in\cV_{\g_{\bK}} L_{p}(\lambda)$.

\end{proof}

For the next result we also need to assume {\bf that there exist algebraically independent homogeneous elements $y_1,\ldots,y_s\in S(\g_R)/J_R(\lambda)$ and elements $v_1,\ldots,v_S\in  S(\g_R)/J_R(\lambda)$ such that} $$S(\g_R)/J_R(\lambda)=R[y_1,\ldots,y_s] v_1 + \cdots + R[y_1,\ldots,y_s] v_S.$$ As in \cite[\S 2.3]{Pr07}, we may make this assumption for a given $\lambda\in\t_\Q^{*}$ by expanding $\cS$ as necessary. We label this assumption ``Assumption (R3($\lambda$))'', noting that it depends on $\lambda\in\t_\Q^{*}$ (and therefore part of the assumption is that $\lambda$ may be defined as element of $\t_R^{*}$).

The following result makes up part of the proof of \cite[Theorem 3.1]{Pr07}. Note that, although our current assumptions are not enough to satisfy all the assumptions of that theorem, they are sufficient for the result we need. This result says as follows.

\begin{prop}\label{prop: irred comp Vp(l)}
Assume that $\lambda\in\t_\Q^{*}$ and that $R$ satisfies assumptions (R1), (R2($\lambda$)) and (R3($\lambda$)) (so in fact $\lambda\in\t_R^{*})$. Then every irreducible component of $\cV_{\g_{\bK}} L_{p}(\lambda)$ of maximal dimension coincides with an irreducible component of $V_p(\lambda)$ of maximal dimension.
\end{prop}

We note for reference that the maximal dimension referenced in each of these propositions is $\frac{1}{2}\dim\O_\lambda$.

\subsection{Type $A$}\label{ss: Type A}

From now on, we only consider the case where $G_\C=\GL_N(\C)$ and we start by only making assumption (R1) (so, in particular, $p\nmid N$). Note that in this setting, for $F\in \{\C,R,\F_p,\bK\}$, we have that $\b_\F=\b_N(\F)\subseteq \gl_N(\F)$ is the Lie subalgebra (or subring) consisting of upper triangular matrices and $\t_\F=\t_N(\F)\subseteq \gl_N(\F)$ is the Lie subalgebra (or subring) consisting of diagonal matrices. We also set $G_\F=\GL_N(\F)$, $B_\F\subseteq G_\F$ as the subgroup of invertible upper triangular matrices with entries in $\F$, and $T_\F\subseteq G_\F$ as the subgroup of invertible diagonal matrices with entries in $\F$.

For each $i=1,\ldots,N$, we define $\varepsilon_i:\t_\F\to \F$ as $$\varepsilon_i:\begin{pmatrix}
a_1 & 0 & \cdots & 0 \\
0 & a_2 & \cdots & 0 \\
\vdots & \vdots & \ddots & \vdots  \\
0 & 0 & \cdots & a_N \\
\end{pmatrix}\mapsto a_i.$$
Denote by $\t_\F^{*}$ the $\F$-module of $\F$-linear maps $\t_\F\to \F$; the elements $\varepsilon_i\in\t_\F^{*}$ form a free $\F$-basis of $\t_\F^{*}$. We may describe the root system $\Phi$ and the subsets of positive and simple roots in $\Phi$ in terms of the $\varepsilon_i$ as follows:
$$\Phi=\{\varepsilon_{i}-\varepsilon_{j}\mid 1\leq i\neq j\leq N\},$$
$$\Phi^{+}=\{\varepsilon_{i}-\varepsilon_{j}\mid 1\leq i< j\leq N\},$$
$$\Pi=\{\varepsilon_{i}-\varepsilon_{i+1}\mid 1\leq i< N\}.$$ Note in particular that $e_{ij}$ is a root vector in $\cB$ corresponding to the root $\varepsilon_i-\varepsilon_j\in\Phi$ (for $i\neq j$). Define $\rho\in\t_\F^{*}$ by $$\rho:=-\varepsilon_1 - 2\varepsilon_2 - \cdots - N\varepsilon_N.$$ The Weyl group $W=S_N$ acts on  $\t_\F^{*}$ in such a way that $\sigma\in S_N$ sends $\varepsilon_i$ to $\varepsilon_{\sigma(i)}$ for all $i=1,\ldots,N$. We also define a dot-action of $W$ on  $\t_\F^{*}$ as $$ w\cdot\lambda = w(\lambda+\rho)-\rho.$$ Note that the $\rho$ considered here is a shift of the $\rho$ defined in Subsection~\ref{ss: Prelims}, but the corresponding dot-actions nevertheless coincide.

Let $\underline{p}=(p_1,p_2,\ldots,p_r)$ be a partition of $N$ with $p_r\geq p_{r-1}\geq \cdots \geq p_1 > 0$. We set $\pi$ to be a left-justified pyramid such that the $i$-th row contains $p_i$ boxes. We label these boxes by the natural numbers $1,\ldots,N$ such that the $p_i$ boxes in the $i$-th row are labelled $p_1+p_2+\cdots + p_{i-1}+1, p_1+p_2+\cdots +p_{i-1}+2,\cdots, p_1+p_2+\cdots + p_{i-1}+p_{i} $ in order. For example, if $\underline{p}=(1,2,2,4)$ then the pyramid $\pi$ is (with labelling): 
\begin{center}
\begin{ytableau}
	1 & \none & \none & \none \\
	2 & 3 & \none & \none \\
	4 & 5 & \none & \none \\
	6 & 7 & 8 & 9 \\
\end{ytableau}
\end{center}
Given a pyramid $\pi$ and a natural number $i\in\{1,\ldots,N\}$, we denote by $\row(i)$ the number of the row containing the $i$-th box (where we count rows from top to bottom, starting with row 1) and by $\col(i)$ the number of the column containing the $i$-th box (where we count columns from left to right, starting with column 1). In the above pyramid, for example, we have $\row(3)=2$, $\row(5)=3$, $\col(5)=2$ and $\col(9)=4$.

We associate to a pyramid $\pi$ the nilpotent element $$e_\pi:=\sum_{\substack{\row(i)=\row(j) \\ \col(i)=\col(j)-1 }}e_{ij}.$$ In the above example, we have $$e_\pi=e_{23} + e_{45} + e_{67}+e_{78}+e_{89}.$$

Note that $\gl_N(\C)$ is equipped with a $\GL_N(\C)$-equivariant symmetric bilinear form $\langle -,-\rangle:\gl_N(\C)\times \gl_N(\C)\to \C $, i.e. the trace form $(x,y)\mapsto \tr(xy)$. This defines a $\GL_n(\C)$-equivariant isomorphism $\gl_N(\C)\to \gl_N(\C)^{*}$ given by $x\mapsto \langle x, -\rangle$. Under this isomorphism $e_\pi$ maps to an element $\chi_\pi\in \gl_N(\C)^{*}$ such that $$\chi_\pi(e_{ij})=\twopartdef{1}{\row(i)=\row(j)\mbox{ and }\col(j)=\col(i)-1,}{0}{\mbox{not.}} $$ In particular, $\chi_\pi$ is in standard Levi form (recalling that the $e_{i, i+1}$ are root vectors for the simple roots for $\gl_N(\C)$).

Given an integral domain $\F$, an {\bf $\F$-filling} of $\pi$ is a function $A:\{1,\ldots,N\}\to \F$. To represent this visually, we view the scalar $a_i:=A(i)$ as being placed into the box of $\pi$ which has label $i$. For example, with $\pi$ as above, we visualise $A:\{1,2,\ldots,9\}\to \F$ as  
\begin{center}
\begin{ytableau}
	a_1 & \none & \none & \none \\
	a_2 & a_3 & \none & \none \\
	a_4 & a_5 & \none & \none \\
	a_6 & a_7 & a_8 & a_9 \\
\end{ytableau}
\end{center}
We write $\Tab_\F(\pi)$ for the set of all $\F$-fillings of $\pi$ (when visualised as above, we often call an $\F$-filling a {\bf tableau}, which explains this notation). Given $A\in\Tab_\F(\pi)$, define $$\lambda_A:=a_1\varepsilon_1 + a_2\varepsilon_2 + \cdots + a_N \varepsilon_N\in\t_\F^{*}.$$

We now list some properties that a tableaux (or pair of tableaux) may have. These of course correspond to properties of $\lambda_A$ in a precise sense, but the visual perspective makes them easier to describe and understand.

We say that $A, A'\in\Tab_{\F}(\pi)$ are {\bf row-equivalent} if $A'$ can be obtained from $A$ by permuting the entries in each row, i.e. if there exists $\sigma\in S_{p_1}\times S_{p_2}\times \cdots\times S_{p_r}$ such that $a_i'=a_{\sigma(i)}$.

If $\F=\C$, which we can equip with a partial order $\leq$ by saying that $a\leq b$ if and only if $b-a\in\Z_{\geq 0}$, then we call $A\in\Tab_\C(\pi)$ {\bf column-strict} if it is strictly increasing going up columns (with respect to this partial order). For example, with $A$ as above this requires $$a_1>a_2>a_4>a_6\quad\mbox{and} \quad a_3>a_5>a_7.$$

We also call $A\in \Tab_\C(\pi)$ {\bf row-standard} if the entries in each row are non-decreasing (with respect to this partial order) from left to right, i.e. for each $k=1,\ldots,n$ and for any $i,j\in\{p_1+\cdots +p_{k-1}+1,\ldots, p_1+\cdots +p_{k} \}$ with $i<j$ we have $a_i\not> a_j$. If $A\in \Tab_\Z(\pi)$, this is equivalent to requiring $$a_{p_1+\cdots +p_{k-1}+1}\leq a_{p_1+\cdots +p_{k-1}+2}\leq \cdots \leq a_{p_1+\cdots +p_{k-1}+p_{k}}$$ for each $k=1,\ldots,r$.

Finally. we call $A\in\Tab_\F(\pi)$ {\bf column-connected} if $a_i=a_j+1$ whenever the box with label $i$ lies directly above the box with label $j$. In the above example, this corresponds to 
$$a_1=a_2+1=a_4+2=a_6+3\quad\mbox{and} \quad a_3=a_5+1=a_7+2.$$

Let us now discuss some of the representation theory of $U_{\chi_\pi}(\g_{\bK})$ through this language. Fix $A\in\Tab_{\F_p}(\pi)$. We define $\bK_A$ to be the 1-dimensional $U_0(\b_{\bK})$-module on which $\n_\bK^{+}$ acts as zero and $\t_\bK$ acts via $\lambda_A-\rho$. We may then define the $U_{\chi_\pi}(\g_\bK)$-module  $Z_{\chi_\pi}(A):=U_{\chi_\pi}(\g_\bK)\otimes_{U_0(\b_{\bK})} \bK_A$; this is just another way of labelling the baby Verma module $Z_{\chi_\pi}(\lambda_A-\rho)$ discussed earlier, but reflecting that in this setting we prefer to write things in terms of tableau. As already observed (recalling that $\chi_\pi$ is in standard Levi form), each baby Verma module $Z_{\chi_\pi}(A)$ has a unique simple quotient which we now denote by $L_{\chi_\pi}(A)$. Each simple $U_{\chi_\pi}(\g_{\bK})$-module is of this form, and 
\begin{equation}\label{e: isom L}
L_{\chi_\pi}(A)\cong L_{\chi_\pi}(A') \quad \iff \quad A\mbox{ and }A'\mbox{ are row-equivalent}.
\end{equation}

Recall from Premet's Theorem \cite{PrKW} that every $U_{\chi_\pi}(\g_\bK)$-module has dimension divisible by $p^{\dim \O_\pi/2}$, where $\O_\pi=G\cdot\chi_\pi$. Due to \cite[Theorem 1.1]{PT21} there is always a $U_{\chi_\pi}(\g_\bK)$-module of dimension exactly $p^{\dim \O_\pi/2}$, which we call a {\bf minimal-dimensional module}. (The result in \cite[Theorem 1.1]{PT21} is proved for all types, but in fact follows easily in type $A$ from the fact that all orbits are Richardson.) The following result due to \cite[Theorem 1.1]{GT} characterises the minimal-dimensional modules for $U_{\chi_\pi}(\gl_N(\bK))$.

\begin{theorem}
The minimal-dimensional $U_{\chi_\pi}(\g_\bK)$-modules are precisely the simple modules $L_{\chi_\pi}(A)$ for those $A\in\Tab_{\F_p}(\pi)$ which are row-equivalent to some column-connected $A'\in\Tab_{\F_p}(\pi)$. 
\end{theorem}

Fix now a minimal-dimensional $U_{\chi_\pi}(\g_\bK)$-module $L$; by the above theorem, there exists $A\in\Tab_{\F_p}(\pi)$ row-equivalent to a column-connected $A'\in\Tab_{\F_p}(\pi)$ such that $L\cong L_{\chi_\pi}(A)$. By (\ref{e: isom L}), we have $L_{\chi_\pi}(A)\cong L_{\chi_\pi}(A')$, and so we may in fact assume that $A$ is itself column-connected.

Our goal now is to lift $A$ to an element $\widehat{A}$ of $\Tab_{\Z}(\pi)$ which has some particular properties. To do so, let the entries of $A$ be $a_1,\ldots,a_N\in \F_p$, which we view as integers lying between $0$ and $p-1$. For the entries of the first column of $\widehat{A}$, choose integers $\hat{a}_1,\hat{a}_{p_1+1},\ldots,\hat{a}_{p_1+p_2+\cdots + p_{r-1}+1}$ which coincide with $a_1,a_{p_1+1},\ldots,a_{p_1+p_2+\cdots + p_{r-1}+1}$ modulo $p$ and which satisfy $$\hat{a}_1=\hat{a}_{p_1+1}+1 = \hat{a}_{p_1+p_2+1}+2 = \cdots = \hat{a}_{p_1+\cdots + p_{r-1}+1}+r-1$$ (this is possible since $A$ is column-connected). To construct the entries for the second column of $\widehat{A}$ we lift the entries of the second column of $A$ in a similar way as for the first column,\footnote{If $p_1\geq 2$, these entries will be $a_2,a_{p_1+2},\cdots,a_{p_1+p_2+\cdots + p_{r-1}+2}$, but the notation will get more convoluted if $p_1=1$ and so going forward we do not write the elements so explicitly.} but we do so in such a way that all entries in the second column of $\widehat{A}$ are greater than all entries in the first column of $\widehat{A}$. Since we may shift all entries in a column of $\widehat{A}$ by a multiple of $p$ without affecting the column-connectedness of $\widehat{A}$, it is straightforward to see that we can indeed do this. By repeating this process to construct each column of $\widehat{A}$ (so that all entries in each column are greater than all entries in all of the preceding columns) we obtain a column-connected and row-standard element of $\Tab_\Z(\pi)$ which restricts to $A$ modulo $p$. In other words, we have lifted $A$ to $\widehat{A}\in\Tab_{\Z}(\pi)$ such that $$\widehat{A} \mbox{ is column-connected and row-standard.}$$ We now always assume that $\widehat{A}$ satisfies these properties.

For example, continuing with our running example and supposing $p=7$, consider the following $A\in\Tab_{\F_p}(\pi)$:
\begin{center}
\begin{ytableau}
	2 & \none & \none & \none \\
	1 & 6 & \none & \none \\
	0 & 5 & \none & \none \\
	6 & 4 & 1 & 0 \\
\end{ytableau}
\end{center}
Such $A$ is column-connected. Applying the above process to form $\widehat{A}$ (noting that such outcome is not unique), we may pick $\widehat{A}$ as 
\begin{center}
\begin{ytableau}
	2 & \none & \none & \none \\
	1 & 13 & \none & \none \\
	0 & 12 & \none & \none \\
	-1 & 11 & 15 & 21 \\
\end{ytableau}
\end{center}

Since $\pi$ is left-justified, $\widehat{A}$ is semi-standard in the language of \cite{Br11}. This means in particular that when one applies the Robinson-Schensted algorithm to the tuple $(\hat{a}_1,\hat{a}_2,\ldots,\hat{a}_N)\in\Z^N$, one gets a left-justified pyramid of shape $\pi$ by \cite[Lemma 3.3]{Br11}. Note that the element of $\t_{\Z}^{*}$ that we refer to as $\lambda_{\widehat{A}}$ is denoted $\rho(\widehat{A})$ in \cite{Br11}, as that paper uses a different convention to number the boxes in a pyramid. Combining \cite[Theorems 2.2, 3.1, and 3.2]{Br11},\footnote{We cite to three theorems in \cite{Br11} for ease of reference and consistency of notation, but each of these theorems predate \cite{Br11}. Indeed, \cite[Theorem 2.2]{Br11} is \cite[Theorem 7.9]{BK} ,\cite[Theorem 3.1]{Br11} is \cite[Theorem 3.1]{PrEnv}, and \cite[Theorem 3.2]{Br11} is a special case of \cite[Theorem 5.1.1]{Lo1}} we conclude that $$\cVA'(I(\lambda_{\widehat{A}}-\rho))=\overline{\O}_\pi.$$  Furthermore, tracing through the argument shows that $I(\lambda_{\widehat{A}}-\rho)$ is completely prime.\footnote{It actually shows (through \cite[Theorem 3.2]{Br11}) that it is a Losev-Premet ideal in the language of \cite{GTW}, and these are automatically completely prime.}

The goal for the remainder of this subsection is to show that $L_p^{\chi_\pi}(\lambda_{\widehat{A}}-\rho)\neq 0$. By Corollary~\ref{cor: gl v sl}, it suffices to show the analogous result over $\sl_N(\bK)$. To avoid cluttering the notation, we continue to write $\chi_\pi$ in place of $\chi_\pi'$ and $\lambda_{\widehat{A}}-\rho$ in place of $(\lambda_{\widehat{A}}-\rho)'$, but from now  on we understand these to represent their respective restrictions to $\sl_N(\C)^{*}$ and $(\t_R')^{*}$. 

Our first step is to show that $$\chi_\pi\in\cV_{\sl_N(\C)} L_{\C}(\lambda_{\widehat{A}}-\rho).$$

\begin{prop}\label{prop: chipi in VL}
Let $\underline{p}$ be a partition of $N$ with associated left-justified pyramid $\pi$, and let $A\in \Tab_{\F_p}(\pi)$ be column-connected. Then, defining $\chi_\pi$ and $\widehat{A}$ as above, we have $\chi_\pi\in\cV_{\sl_N(\C)} L_{\C}(\lambda_{\widehat{A}}-\rho)$.
\end{prop}

\begin{proof}

By standard results on complex semisimple Lie algebras, there exists an element $w\in W$ and a dominant integral weight $\lambda_{\widehat{A},0}\in(\t_\C')^{*}$ such that the following two properties hold: first, the shift $\lambda_{\widehat{A},0}-\rho\in(\t_\C')^{*}$ remains dominant and integral and, second, there is equality $$\lambda_{\widehat{A}}-\rho=w\cdot(\lambda_{\widehat{A},0}-\rho)= w(\lambda_{\widehat{A},0}) - \rho.$$ By our constructions, there exists $\widehat{A}^0\in\Tab_\Z(\pi)$ such that $\lambda_{\widehat{A}^0}=\lambda_{\widehat{A},0}$; in other words, $\lambda_{\widehat{A},0}=\hat{a}^0_1 \varepsilon_1 + \cdots + \hat{a}^0_N \varepsilon_N$ where $\hat{a}^0_1,\ldots,\hat{a}^0_N\in\Z$ are the entries of $\widehat{A}^0$. The property that $\lambda_{\widehat{A},0}-\rho$ is dominant then corresponds to $$\hat{a}_1^0> \hat{a}_2^0> \cdots > \hat{a}_N^0.$$ By \cite[Theorem 8.15]{Jo84} and \cite[Corollary 4.3.1]{BoBr}, identifying $\sl_N(\C)^{*}$ with $\sl_N(\C)$ (via the trace form\footnote{We use the trace form for $\sl_N(\C)$ rather than the Killing form so that it is compatible with the bilinear form we used on $\gl_N(\C)$. Note that nothing we did in earlier sections depended on the precise non-degenerate symmetric bilinear form we were using. Note also, however, that the trace form on $\sl_N(\bK)$ ceases to be non-degenerate when $p\mid N$ and therefore our requirement that $R$ satisfy assumption (R1) is necessary here.}) and identifying $\cV_{\sl_N(\C)} L_{\C}(\lambda)$ with its corresponding subvariety of $\sl_N(\C)$, we have
$$\cV_{\sl_N(\C)} L_{\C}(\lambda) \supseteq \overline{\Ad(\widetilde{B})(\n_\C^{+}\cap ww_0(\n_\C^{+}))},$$ where $$ww_0(\n_\C^{+})=\bigoplus_{\alpha\in \Phi^{+}} \g_{\C,ww_0(\alpha)} $$ and $\widetilde{B}=B_\C\cap\SL_N(\C)$. Thus $\chi_\pi\in \cV_{\sl_N(\C)} L_{\C}(\lambda)$ if $ e_\pi\in \overline{\Ad(\widetilde{B})(\n_\C^{+}\cap ww_0(\n_\C^{+}))}$.
By construction $e_\pi\in\n_\C^{+}$ and it thus suffices to show that $e_\pi\in ww_0(\n_\C^{+})$, i.e. $\Ad(\dot{w}_0\dot{w}^{-1})(e_\pi)\in\n_\C^{+}$. Recall that $\dot{w}$ and $\dot{w}_0$ are permutation matrices lifting the elements $w, w_0\in S_N$; setting $\sigma=w_0w^{-1}\in S_N$, the product $\dot{w}_0 \dot{w}^{-1}$ is then the permutation matrix $\dot{\sigma}$ representing $\sigma$. Note that $\Ad(\dot{\sigma})(e_{ij})=e_{\sigma(i),\sigma(j)}$ for all $1\leq i,j\leq N$.

Recall that $$e_\pi=\sum_{\substack{\row(i)=\row(j) \\ \col(i)=\col(j)-1 }}e_{ij};$$ since we count boxes along rows first, each summand of this element is of the form $e_{i,i+1}$ for some $i\in\{1,\ldots, N-1\}$. Thus each summand of $$\Ad(\dot{w}_0\dot{w}^{-1})(e_\pi)=\sum_{\substack{\row(i)=\row(j) \\ \col(i)=\col(j)-1 }}\Ad(\dot{w}_0\dot{w}^{-1})(e_{ij})$$ is of the form $e_{\sigma(i),\sigma(i+1)}$ for some $i\in\{1,\ldots, N-1\}$ with $\row(i)=\row(i+1)$ and $\col(i+1)=\col(i)+1$. Fix such an $i$. Note now that $$w^{-1}(i)>w^{-1}(i+1)$$ for each $i$ in this summand since the entries of $\widehat{A}$ are increasing along rows. Furthermore, as $w_0$ sends $j$ to $N-j+1$ for each $j\in\{1,\ldots,N\}$, we have $w_0w^{-1}(i)< w_0w^{-1}(i+1)$, i.e. $\sigma(i)<\sigma(i+1)$. In particular, $e_{\sigma(i),\sigma(i+1)}\in\n_\C^{+}$ and thus $$\Ad(\dot{w}_0\dot{w}^{-1})(e_\pi)\in\n_\C^{+}.$$\end{proof}

Since $\chi$ is defined over $\Z$, if we temporarily make assumption (R2($\lambda_{\widehat{A}}-\rho$)) then we may apply Proposition~\ref{prop: R to K in VL} to get the following.

\begin{cor}\label{cor: chipi in VL mod p}
Let $\underline{p}$ be a partition of $N$ with associated left-justified pyramid $\pi$, and let $A\in \Tab_{\F_p}(\pi)$ be column-connected. Define $\chi_\pi$ and $\widehat{A}$ as above, and assume that $R$ satisfies assumptions (R1) and (R2($\lambda_{\widehat{A}}-\rho$)) (for $\g_R=\sl_N(R))$). Then we have 
$$\chi_\pi\otimes 1\in\cV_{\sl_N(\bK)}L_p(\lambda_{\widehat{A}}-\rho).$$
\end{cor}

Following Proposition~\ref{prop: irred comp Vp(l)}, our next step is therefore to show that $\chi_\pi\otimes 1$ in fact lies in an irreducible component of $\cV_{\sl_N(\bK)}L_p(\lambda_{\widehat{A}}-\rho)$ of maximal dimension. We do this in two stages: first, we show that $\cV_{\sl_N(\bK)}L_p(\lambda_{\widehat{A}}-\rho)$ is closed under a certain group action, and second we show that the dimension of the associated orbit of $\chi_\pi\otimes 1$ coincides with the dimension of $\cV_{\sl_N(\bK)}L_p(\lambda_{\widehat{A}}-\rho)$.

The group that we wish to use is essentially the Borel subgroup $B_\bK$ of $\GL_N(\bK)$; however, since at present we are considering everything for $\SL_N(\bK)$ and since we will need to use some scheme-theoretic arguments in the proofs, we introduce the relevant group (scheme) in a slightly different way.

Specifically, we define $\widetilde{B}_\F$ to be the $\F$-group scheme associated to the functor $$A\mapsto \left\{\begin{pmatrix}
a_{11} & a_{12} & \cdots & a_{1N} \\
0 & a_{22} & \cdots & a_{2N} \\
\vdots & \vdots & \ddots & \vdots  \\
0 & 0 & \cdots & a_{NN} \\
\end{pmatrix} \mid a_{ij}\in A \mbox{ and } a_{11}a_{22}\cdots a_{NN}=1\right\}$$ from the category of (finitely-generated) $\F$-algebras to the category of groups.

This group scheme is represented by the $\F$-Hopf algebra $$\F[\widetilde{B}_\F] := \F[x_{ij}\mid 1\leq i \leq j\leq N]/\langle x_{11}x_{22}\cdots x_{NN}-1\rangle, $$ with Hopf algebra structure induced from that on $\F[\SL_N]$.

Note that $\widetilde{B}_\F(A)$ acts on $\sl_N(A)$ via conjugation for any $\F$-algebra $A$ and this makes $\sl_N(\F)$ into a $\widetilde{B}_\F$-module. In particular, this means that there exists a map
$$\Delta_{\sl}: \sl_N(\F)\to\F[\widetilde{B}_\F]\otimes_\F \sl_N(\F) $$ which makes $\sl_N(\F)$ into an $\F[\widetilde{B}_\F]$-comodule. This comodule structure can then be extended to a map $$\Delta_{\sl}: U(\sl_N(\F))\to\F[\widetilde{B}_\F]\otimes_\F U(\sl_N(\F)),$$ thereby turning $U(\sl_N(\F))$ into an $\F[\widetilde{B}_\F]$-comodule algebra. 

To state the following proposition, recall that the algebraic group $\SL_N(\bK)$ also acts on $\sl_N(\bK)^{*}$ via the coadjoint action (and that this corresponds to an action of $\SL_N(\bK)$ on $S(\sl_N(\bK))$ by algebra homomorphisms).

\begin{prop}\label{prop: B stab VL}
The subvariety $\cV_{\sl_N(\bK)}L_p(\lambda_{\widehat{A}}-\rho)$ of $\sl_N(\bK)^{*}$ is $\widetilde{B}_{\bK}(\bK)$-stable.
\end{prop}

\begin{proof}

By construction, it is sufficient to show that $\Ann_{S(\sl_N(\bK))}\gr L_p(\lambda_{\widehat{A}}-\rho)$ is a $\widetilde{B}_\bK(\bK)$-stable ideal of $S(\sl_N(\bK))$. This will follow easily if $\gr L_p(\lambda_{\widehat{A}}-\rho)$ can be equipped with the structure of a $\widetilde{B}_\bK(\bK)$-module in such a way that the module map $$S(\sl_N(\bK))\otimes_{\bK} \gr L_p(\lambda_{\widehat{A}}-\rho) \to \gr L_p(\lambda_{\widehat{A}}-\rho)$$ is $\widetilde{B}_\bK(\bK)$-equivariant. 

To define such a module structure, we begin by observing that we can equip the Verma module $M_{\C}(\lambda_{\widehat{A}}-\rho)$ with the structure of a $\C[\widetilde{B}_\C]$-comodule $$\Delta:M_{\C}(\lambda_{\widehat{A}}-\rho) \to \C[\widetilde{B}_\C]\otimes M_{\C}(\lambda_{\widehat{A}}-\rho)$$ such that the following diagram commutes
\begin{eqnarray}
	\label{e: Verma comod}
	\begin{array}{c}\xymatrix{
			U(\g_\C)\otimes_\C M_\C(\lambda_{\widehat{A}}-\rho) \ar@{->}[rr]^{\widetilde{\Delta}} \ar@{->}[d]^{m} & & \C[\widetilde{B}_\C]\otimes_\C(U(\g_\C)\otimes_\C M_\C(\lambda_{\widehat{A}}-\rho)) \ar@{->}[d]^{1\otimes m}\\
			M_\C(\lambda_{\widehat{A}}-\rho) \ar@{->}^{\Delta}[rr] & & \C[\widetilde{B}_\C]\otimes_\C M_\C(\lambda_{\widehat{A}}-\rho). 
		}
	\end{array}
\end{eqnarray}
Here, the first horizontal map is $$\widetilde{\Delta}=(m_{\C[\widetilde{B}_\C]}\otimes 1\otimes 1)\circ(1\otimes \sigma\otimes 1)\circ (\Delta_\sl \otimes \Delta)$$ where $m_{\C[\widetilde{B}_\C]}$ is the multiplication on $\C[\widetilde{B}_\C]$ and $\sigma$ is the swapping map. In particular, the comodule structure is defined so that $$\Delta(v_{\lambda_{\widehat{A}}-\rho})=x_{11}^{(\lambda_{\widehat{A}}-\rho)(e_{11})} \cdots x_{NN}^{(\lambda_{\widehat{A}}-\rho)(e_{NN})}\otimes  v_{\lambda_{\widehat{A}}-\rho}$$ (recalling that $(\lambda_{\widehat{A}}-\rho)(e_{ii})\in \Z$ for all $i=1,\ldots,N$).

Furthermore, this $\C[\widetilde{B}_\C]$-comodule structure induces a $\widetilde{B}_\C$-module structure on $M_\C(\lambda_{\widehat{A}}-\rho)$ with the property that the derivative of the action coincides with the restriction of the $\sl_N(\C)$-module structure to $\widetilde{\b}_\C=\Lie(\widetilde{B}_\C)$. Since $M_\C^{\max}(\lambda_{\widehat{A}}-\rho)$ is a $U(\widetilde{\b}_\C)$-submodule and $U(\widetilde{\b}_\C)=\Dist(\widetilde{B}_\C)$, we also have that $M_\C^{\max}(\lambda_{\widehat{A}}-\rho)$ is a $\widetilde{B}_\C$-submodule of $M_\C(\lambda_{\widehat{A}}-\rho)$ (see \cite[I.7.15]{JanRAGS} and also \cite[I.7.17(6)]{JanRAGS}). The simple highest weight module $L_\C(\lambda_{\widehat{A}}-\rho)$ may therefore also be equipped with a $\widetilde{B}_\C$-module structure and, equivalently, a $\C[\widetilde{B}_\C]$-comodule structure such that a diagram analogous to (\ref{e: Verma comod}) is commutative.

In particular, for $c_1,\ldots,c_D\geq 0$, we thus get
\begin{equation*}
\begin{split}\Delta(e_{-\gamma_D}^{c_D}\cdots e_{-\gamma_1}^{c_1}\overline{v}_{\lambda_{\widehat{A}}-\rho}) & =\Delta\circ m(e_{-\gamma_D}^{c_D}\cdots e_{-\gamma_1}^{c_1}\otimes \overline{v}_{\lambda_{\widehat{A}}-\rho})\\ & = (1\otimes m)\circ \widetilde{\Delta}(e_{-\gamma_D}^{c_D}\cdots e_{-\gamma_1}^{c_1}\otimes \overline{v}_{\lambda_{\widehat{A}}-\rho})\\
	& = (1\otimes m)\circ(m_{\C[\widetilde{B}_\C]}\otimes 1\otimes 1)\circ(1\otimes \sigma\otimes 1)\circ (\Delta_\sl(e_{-\gamma_D}^{c_D}\cdots e_{-\gamma_1}^{c_1}) \otimes \Delta(\overline{v}_{\lambda_{\widehat{A}}-\rho})).
\end{split}
\end{equation*}

Since the conjugation action of $\SL_N(\C)$ on $\sl_N(\C)$ restricts to the conjugation action of $\SL_N(R)$ on $\sl_N(R)$, the following diagram commutes
\begin{eqnarray}
\label{e: pidiagprop1}
\begin{array}{c}\xymatrix{
		U(\g_R) \ar@{->}^{\Delta_{\sl}}[rrr]\ar@{^{(}->}[d] & & & R[\widetilde{B}_R]\otimes_R U(\g_R) \ar@{^{(}->}[d]\\
		U(\g_\C) \ar@{->}[rrr]^{\Delta_{\sl}} & & & \C[\widetilde{B}_\C]\otimes_\C U(\g_\C),
	}
\end{array}
\end{eqnarray}
which implies that $$\Delta_\sl(e_{-\gamma_N}^{c_D}\cdots e_{-\gamma_1}^{c_1})\in R[\widetilde{B}_R]\otimes_R U(\sl_N(R)).$$ Furthermore, from the $\widetilde{B}_\C$-module structure on $L_\C(\lambda_{\widehat{A}}-\rho)$ we have $$\Delta(\overline{v}_{\lambda_{\widehat{A}}-\rho})=x_{11}^{(\lambda_{\widehat{A}}-\rho)(e_{11})} \cdots x_{NN}^{(\lambda_{\widehat{A}}-\rho)(e_{NN})}\otimes  \overline{v}_{\lambda_{\widehat{A}}-\rho} \in R[\widetilde{B}_R]\otimes M_R(\lambda_{\widehat{A}}-\rho).$$

We therefore conclude that $$\Delta(e_{-\gamma_D}^{c_D}\cdots e_{-\gamma_1}^{c_1}\overline{v}_{\lambda_{\widehat{A}}-\rho})\in R[\widetilde{B}_R]\otimes L_R(\lambda_{\widehat{A}}-\rho) $$ for all $c_1,\ldots,c_D\geq 0$; in particular, recalling that $L_R(\lambda_{\widehat{A}}-\rho)=U(\g_R)\overline{v}_{\lambda_{\widehat{A}}-\rho}\subseteq L_\C(\lambda_{\widehat{A}}-\rho)$, this means that $\Delta$ restricts to a map $$\Delta:L_R(\lambda_{\widehat{A}}-\rho)\to R[\widetilde{B}_R]\otimes_R L_R(\lambda_{\widehat{A}}-\rho).$$    
By base change, we then get a $\bK[\widetilde{B}_\bK]$-comodule structure
$$\Delta: L_p(\lambda_{\widehat{A}}-\rho)=L_R(\lambda_{\widehat{A}}-\rho)\otimes_R \bK \to (R[\widetilde{B}_R]\otimes_R L_R(\lambda_{\widehat{A}}-\rho))\otimes_R \bK = \bK[\widetilde{B}_\bK]\otimes_\bK L_p(\lambda_{\widehat{A}}-\rho);$$ i.e. a $\widetilde{B}_\bK$-module structure on $L_p(\lambda_{\widehat{A}}-\rho)$. From the construction, it is clear that the following diagram commutes:
\begin{eqnarray}
\label{e: equiv2}
\begin{array}{c}\xymatrix{
		U(\g_\bK)\otimes_\bK L_p(\lambda_{\widehat{A}}-\rho) \ar@{->}[rr]^{\widetilde{\Delta}} \ar@{->}[d]^{m} & & \bK[\widetilde{B}_\bK]\otimes_\bK(U(\g_\bK)\otimes_\bK L_p(\lambda_{\widehat{A}}-\rho)) \ar@{->}[d]^{1\otimes m}\\
		L_p(\lambda_{\widehat{A}}-\rho) \ar@{->}^{\Delta}[rr] & & \bK[\widetilde{B}_\bK]\otimes_\bK L_p(\lambda_{\widehat{A}}-\rho). 
	}
\end{array}
\end{eqnarray}

Since $L_p(\lambda_{\widehat{A}}-\rho)=U(\g_{\bK})\overline{w}_{\lambda_{\widehat{A}}-\rho}$ and the filtration is defined so that $L_{p,n}(\lambda_{\widehat{A}}-\rho)=U_n(\g_{\bK})\overline{w}_{\lambda_{\widehat{A}}-\rho}$, it is clear from the above diagram that the action of $\widetilde{B}_\bK(\bK)$ preserves the filtration. Indeed, the adjoint action of $\widetilde{B}_\bK(\bK)$ on $U(\g_\bK)$ preserves the PBW filtration and from construction it is clear that $\widetilde{B}_\bK(\bK) \overline{w}_{\lambda_{\widehat{A}}-\rho}\subseteq \bK \overline{w}_{\lambda_{\widehat{A}}-\rho}$. Therefore, this induces an action of $\widetilde{B}_\bK(\bK)$ on $\gr L_p(\lambda_{\widehat{A}}-\rho)$.

The diagram (\ref{e: equiv2}) gives the $\widetilde{B}_\bK(\bK)$-equivariance of the map $$U(\g_\bK)\otimes_\bK L_p(\lambda_{\widehat{A}}-\rho)\to L_p(\lambda_{\widehat{A}}-\rho),$$ which implies the $\widetilde{B}_{\bK}(\bK)$-equivariance of the map $$S(\g_\bK)\otimes_{\bK} \gr L_p(\lambda_{\widehat{A}}-\rho)\to \gr L_p(\lambda_{\widehat{A}}-\rho).$$ This concludes the proof.

\end{proof}

\begin{prop}\label{prop: dim B orb}
The subvariety $\widetilde{B}_\bK(\bK)\cdot(\chi_\pi\otimes 1)$ has dimension $\frac{1}{2}\dim\O_\pi$.
\end{prop}

\begin{proof}

Using the isomorphism $\sl_N(\bK)\cong \sl_N(\bK)^{*}$, it will be enough for us to compute $\dim(\widetilde{B}_\bK(\bK)\cdot (e_\pi\otimes 1))$. (For ease of notation, we write $e_\pi$ in place of $e_\pi\otimes 1$ for the remainder of this proof.) Recalling that $$B_{\bK}=\left\{\begin{pmatrix}
a_{11} & a_{12} & \cdots & a_{1N} \\
0 & a_{22} & \cdots & a_{2N} \\
\vdots & \vdots & \ddots & \vdots  \\
0 & 0 & \cdots & a_{NN} \\
\end{pmatrix} \mid a_{ij}\in \bK \mbox{ and } a_{11}\cdots a_{NN}\neq 0 \right\}\subseteq \GL_N(\bK),$$ it is elementary that $\dim(\widetilde{B}_\bK(\bK)\cdot e_\pi)= \dim(B_\bK\cdot e_\pi)$, where the latter is a subvariety of $\gl_N(\bK)$. We compute the latter dimension, since it is easier. In what follows, we write $B$ for $B_\bK$ for ease of notation.

Note that $\dim (B\cdot e_\pi) = \dim B - \dim B_{e_\pi}$ and that $$B_{e_\pi}=B\cap \gl_N(\bK)_{e_\pi}\subseteq \gl_N(\bK),$$ where $B_{e_\pi}$ denotes the stabiliser of $e_\pi$ in $B$ and $\gl_N(\bK)_{e_\pi}$ denotes the centraliser of $e_\pi$ in $\gl_N(\bK)$. For $i=1,\ldots,r$ set $v_i$ to be the standard basis element of $\bK^N$ with a $1$ in the $(p_1+\cdots + p_i)$-th entry and zeroes elsewhere. The standard basis of $\bK^N$ can then be alternately viewed as $$\{e_\pi^k(v_i) \mid 1\leq i\leq r,\,\, 0\leq k < p_i\}.$$ As in \cite[\S 3.1]{JanNO}, an element $Z\in \gl_N(\bK)_{e_\pi}$ is uniquely determined by its effect on $v_i$ for each $i=1,\ldots,r$, since $Z(e_{\pi}^k(v_i))=e_{\pi}^k(Z(v_i))$ for all $k$. Furthermore, each $Z(v_i)$ can be written as $$Z(v_i)=\sum_{j=1}^{r} \sum_{k=\max(0,p_j-p_i)}^{p_j-1} a_{k,j;i} e_\pi^k(v_j)$$ with the $a_{k,j;i}$ chosen arbitrarily. 

Writing this in matrix form, it is easy to see that such $Z$ lies in $B$ if and only if $a_{0,1;1}a_{0,2;2}\cdots a_{0,r;r}\neq 0$ and $a_{k,j;i}=0$ whenever $j>i$. In particular, this means that $$\dim B_{e_\pi}=\sum_{1\leq j\leq i\leq r}\min(p_i,p_j).$$ 

Note that, as in \cite[\S 3.1]{JanNO}, $$\dim \gl_N(\bK)_{e_\pi}=\sum_{1\leq  i,j\leq r}\min(p_i,p_j).$$ Therefore $$\dim \gl_N(\bK)_{e_\pi} = \sum_{i=1}^{r} p_i + 2\sum_{1\leq j<i\leq r} \min(p_i,p_j)=N+2\sum_{1\leq j<i\leq r} \min(p_i,p_j).$$ This implies $$\sum_{1\leq j<i\leq r} \min(p_i,p_j) =\frac{1}{2} \left( \dim \gl_N(\bK)_{e_\pi} - N \right)$$ and thus 
$$\dim B_{e_\pi} = \sum_{1\leq j\leq i\leq r}\min(p_i,p_j)=\frac{1}{2} \left( \dim \gl_N(\bK)_{e_\pi} + N \right).$$ Hence $$\dim B - \dim B_{e_\pi} = \frac{1}{2}(2\dim B - N - \dim \gl_N(\bK)_{e_\pi})=\frac{1}{2}(\dim \gl_N(\bK) - \dim \gl_N(\bK)_{e_\pi})$$ and so $$\dim (B\cdot e_\pi) = \frac{1}{2}\dim \O_\pi$$ as required.
\end{proof}

We may put all this together into the following corollary. To state it, recall the definition of $V_p(\lambda_{\widehat{A}}-\rho)\subseteq \sl_N(\bK)^{*}$ from Subsection~\ref{ss: Assoc Var}. We now make assumptions (R1), (R2($\lambda_{\widehat{A}}-\rho$)) and (R3($\lambda_{\widehat{A}}-\rho$)) for $\g_R=\sl_N(R)$, as they are necessary now to apply (amongst others) Proposition~\ref{prop: irred comp Vp(l)}.

\begin{cor}\label{cor: chi in Vp}
Let $\underline{p}$ be a partition of $N$ with associated left-justified pyramid $\pi$, and let $A\in \Tab_{\F_p}(\pi)$ be column-connected. Define $\chi_\pi$ and $\widehat{A}$ as above, and assume that $R$ satisfies assumptions (R1), (R2($\lambda_{\widehat{A}}-\rho$)) and (R3($\lambda_{\widehat{A}}-\rho$)) (for $\g_R=\sl_N(R))$). Then we have 
$$\chi_\pi\otimes 1\in V_p(\lambda_{\widehat{A}}-\rho).$$
\end{cor}

\begin{proof}

Since Corollary~\ref{cor: chipi in VL mod p} implies that $\chi_\pi\otimes 1\in \cV_{\sl_N(\bK)}L_p(\lambda_{\widehat{A}}-\rho)$ and Proposition~\ref{prop: B stab VL} implies that $\cV_{\sl_N(\bK)}L_p(\lambda_{\widehat{A}}-\rho)$ is $\widetilde{B}_\bK(\bK)$-stable, we have $$\widetilde{B}_{\bK}(\bK)\cdot(\chi_\pi\otimes 1)\subseteq \cV_{\sl_N(\bK)}L_p(\lambda_{\widehat{A}}-\rho). $$ Since $\widetilde{B}_{\bK}(\bK)\cdot(\chi_\pi\otimes 1)$ is irreducible, it lies in an irreducible component of $\cV_{\sl_N(\bK)}L_p(\lambda_{\widehat{A}}-\rho)$. By Proposition~\ref{prop: dim B orb} and \cite[3.3(2)]{Pr07} $$\dim (\widetilde{B}_\bK(\bK)\cdot (\chi_\pi\otimes 1)) = \frac{1}{2} \dim \O_\pi = \dim \cV_{\sl_N(\bK)}L_p(\lambda_{\widehat{A}}-\rho);$$ this implies that $\widetilde{B}_\bK(\bK)\cdot(\chi_\pi\otimes 1)$ lies inside an irreducible component of $\cV_{\sl_N(\bK)}L_p(\lambda_{\widehat{A}}-\rho)$ of maximal dimension. By Proposition~\ref{prop: irred comp Vp(l)}, this implies that $\chi_\pi\otimes 1$ lies in $V_p(\lambda_{\widehat{A}}-\rho)$.
\end{proof}

Combining Corollary~\ref{cor: gl v sl}, Lemma~\ref{lem: Vp Lchi n0} and Corollary~\ref{cor: chi in Vp} then yields the following corollary.

\begin{cor}
The $U_{\chi_\pi}(\gl_N(\bK))$-module $L_p^{\chi_\pi}(\lambda_{\widehat{A}}-\rho)$ is non-zero.
\end{cor}

Together with Corollary~\ref{cor: surj hom simple}, this then gives the theorem we desired.

\begin{theorem}\label{thm: summation}
Let $\underline{p}$ be a partition of $N$ with associated left-justified pyramid $\pi$. Let $L$ be a minimal-dimensional $U_{\chi_\pi}(\g_\bK)$-module. Then there exists $\lambda\in \t_\Q^{*}$ such that 
\begin{enumerate}
\item  $L\cong L_\chi(\widetilde{\lambda})$,
\item $I(\lambda)$ is completely prime, and
\item $\cVA'(I(\lambda))=\overline{\O}_\pi$.
\end{enumerate}
Furthermore, if there exists $R=\cS^{-1}\Z$ which satisfies (R1), (R2($\lambda$)),  (R3($\lambda$)) (so $\lambda\in\t_R^{*}$) and that $p$ is invertible in $R$, then there exists a surjection $L_p^{\chi_\pi}(\lambda)\twoheadrightarrow L$.  
\end{theorem}

\end{document}